\documentclass[12pt,leqno]{article}
\usepackage[mathscr]{eucal}
\usepackage{amsmath,amssymb,latexsym,theorem,bbm}
\usepackage[english]{babel}
\usepackage{color,url}
\usepackage{soul}

\newcommand{\CC}{\mathbb{C}}

\newcommand{\NN}{\mathbb{N}}

\newcommand{\RR}{\mathbb{R}}

\newcommand{\ZZ}{\mathbb{Z}}

\newcommand{\cA}{{\mathcal A}}

\newcommand{\cD}{{\mathcal D}}

\newcommand{\cF}{{\mathcal F}}

\newcommand{\cN}{{\mathcal N}}

\newcommand{\dd}{\mathrm{d}}
\newcommand{\ee}{\mathrm{e}}
\newcommand{\ii}{\mathrm{i}}

\newcommand{\EE}{\operatorname{\mathbb{E}}}

\newcommand{\PP}{\operatorname{\mathbb{P}}}

\newcommand{\sign}{\operatorname{sign}}

\renewcommand{\Re}{{\operatorname{Re}}}
\renewcommand{\Im}{{\operatorname{Im}}}

\newcommand{\vare}{\varepsilon}

\renewcommand{\mid}{\,|\,}

\renewcommand{\leq}{\leqslant}
\renewcommand{\geq}{\geqslant}

\newcommand{\distr}{\stackrel{\cD}{\longrightarrow}}
\newcommand{\distre}{\stackrel{\cD}{=}}

\newcommand{\bbone}{\mathbbm{1}}

\newcommand{\proofend}{\hfill\mbox{$\Box$}}

\numberwithin{equation}{section}

\theoremstyle{change} \theorembodyfont{\em}
\newtheorem{Lem}{Lemma.}[section]
\newtheorem{Thm}[Lem]{Theorem.}
\newtheorem{Pro}[Lem]{Proposition.}

\newtheorem{Def}[Lem]{Definition.}

\theorembodyfont{\rm}
\newtheorem{Rem}[Lem]{Remark.}
\newtheorem{Ex}[Lem]{Example.}

\def\OnlyOnArXiv#1#2{\ifthenelse{\equal{#1}{Y}}{#2}{}}

\newenvironment{proof}{\noindent{\bf Proof.}}{\proofend}

\allowdisplaybreaks

\begin{document}

\begin{center}
 {\bfseries\Large On variation functions of lacunary series}

\vspace*{3mm}

{\sc\large
  M\'aty\'as $\text{Barczy}^{*,\diamond}$,
  Peter $\text{Kern}^{**}$ }

\end{center}

\vskip0.2cm

\noindent
 * HUN-REN–SZTE Analysis and Applications Research Group,
   Bolyai Institute, University of Szeged,
   Aradi v\'ertan\'uk tere 1, H--6720 Szeged, Hungary.

\noindent
 ** Faculty of Mathematics and Natural Sciences, Heinrich Heine University D\"usseldorf, Universit\"atsstra{\ss}e 1, 
    D-40225 D\"usseldorf, Germany.

\noindent E-mails: barczy@math.u-szeged.hu (M. Barczy),
                   kern@hhu.de (P. Kern).

\noindent $\diamond$ Corresponding author.

\vskip0.2cm

\renewcommand{\thefootnote}{}
\vspace*{0.2cm}
\footnote{A part of this research was carried out when M\'aty\'as Barczy visited Heinrich Heine University D\"usseldorf in June and July, 2024
 thanks to a scholarship in the programme Research Stays for University Academics and Scientists, 2024,
 granted by the German Academic Exchange Service (DAAD).}

\vspace*{-10mm}

\begin{abstract}
We study variation functions of some lacunary series
 along the sequence of $b$-adic partitions, where $b>1$ is an integer.
By such a lacunary series, we mean that in the definition of the well-known Weierstrass function, 
 the Lipschitz continuous cosine and sine functions are replaced by a continuous and piecewise continuously differentiable 
 periodic function whose Fourier series has zero $(jb)^{\mathrm{th}}$-Fourier coefficients for all non-zero integers $j$.
In the proofs, we use a probabilistic approach applying a martingale central limit theorem. 
Our result extends some recent results of Han, Schied and Zhang (2021), 
 and can be applied when one chooses the periodic function in question, 
 for example, as the Takagi function, a Weierstrass function or a certain lacunary series itself. 
\end{abstract}


\section{Introduction}
\label{section_intro}

Given a periodic function, studying connections between its Fourier series and its $p$-variation,
where $p\geq 1$, has a long history,
 at least it goes back to Young \cite{You}, who studied  relationships between Fourier series and functions of bounded variation.
Later, Wiener \cite{Wie} investigated the relationship between the quadratic variation of a periodic function and the convergence of its Fourier series approximation.
For Marcinkiewicz's result on the order of Fourier coefficients of a $2\pi$-periodic function 
 having bounded $p$-variation on $[0,2\pi]$, where $p\geq 1$, see, e.g., 
 Dudley and Norvai\v{s}a \cite[Theorem 11.2]{DudNor}.
Without claiming to completeness, we just mention the more recent papers of
 F\"ul\"op and M\'oricz \cite{FulMor} and Ghodadra \cite{Gho} on the order of magnitude of multiple Fourier coefficients of functions
 of bounded variation, and  of bounded $p$-variation having lacunary Fourier series, respectively.
Motivated by Marcinkiewicz's result, Manstavi\v{c}ius \cite{Man} computed the Fourier coefficients of the linear fractional stable motion  
 and of the closely related Riemann–Liouville process, and investigated the rate of their decay
 with the intention of finding the values of $p$ for which the stochastic processes in question
 are of bounded $p$-variation. 
In this paper, instead of $p$-variation, we investigate $\Theta$-variation of some lacunary series that are special Weierstrass-type functions
 along the sequence of $b$-adic partitions, where $b>1$ is an integer and $\Theta$ is the function given in \eqref{def_Theta}.
Our result has been motivated by and extends some recent results of Han et al.\ \cite{HanSchZha} on the Takagi and Weierstrass functions.
In the formulation of the assumptions of our result and in its proof, Fourier series play a crucial role.

Throughout this paper, let $\NN$, $\ZZ_+$, $\ZZ$, $\RR$, $\RR_+$, $\RR_{++}$ and $\CC$ denote the sets of positive integers,
 non-negative integers, integers, real numbers, non-negative real numbers, positive real numbers and complex numbers, respectively.
The indicator function of a subset $A$ of $\RR$ is denoted by $\bbone_A$.
The Dirac measure concentrated at a point $x\in\RR$ is denoted by $\delta_{x}$.
An interval in $\RR$ will be called nondegenerate if it contains at least two distinct points.
All the random variables will be defined on a common probability space $(\Omega,\cA,\PP)$.
Equality and convergence of distributions of random variables is denoted by $\distre$ 
and $\distr$, respectively.

\begin{Def}\label{Def_cont_Theta_variation}
Let $\Theta:[0,1)\to\RR_+$ be a continuous and increasing
 function with $\Theta(0)=0$ and let $f:[0,1]\to\RR$ be a continuous function. 
For a fixed $b\in\NN\setminus \{1\}$, let
 $\Pi_n:=\{jb^{-n}: j=0,1,\ldots,b^n\}$, $n\in\NN$, be the (refining) sequence of $b$-adic partitions of $[0,1]$.
If there exists a continuous function $\langle f \rangle^{\Theta}:[0,1]\to \RR_+$ such that
 \begin{align}\label{Theta_var_Def}
  V^{\Theta,t}_n(f) := \sum_{j=0}^{\lfloor tb^n\rfloor} \Theta(\vert f((j+1)b^{-n}) - f(jb^{-n})\vert)
      \to \langle f \rangle^{\Theta}(t) \qquad \text{as \ $n\to\infty$}
 \end{align}
 for all $t\in[0,1]$, then the function $\langle f \rangle^{\Theta}$ is said to be the continuous $\Theta$-variation function
 of $f$ along the sequence of $b$-adic partitions $\Pi_n$, $n\in\NN$.
\end{Def}

In the sum $V^{\Theta,t}_n(f)$ in \eqref{Theta_var_Def}, the function $f$ defined on $[0,1]$ is formally evaluated at $1+b^{-n}>1$ if $t=1$ and $j=b^n$.
To handle this, we assume here and in the sequel that, for a function
 $f$ defined on $[0,1]$, we extended the domain of $f$ to $\RR_+$ by setting
 $f(t):=f(1)$ for $t>1$.
Note also that, in Definition \ref{Def_cont_Theta_variation}, the function $f$ is uniformly continuous and hence
 there exists $n_0\in\NN$ such that $\vert f((j+1)b^{-n}) - f(jb^{-n})\vert<1$ for all $n\geq n_0$ and
 $j\in\{0,\ldots,b^n\}$. 
Thus $\Theta(\vert f((j+1)b^{-n}) - f(jb^{-n})\vert)$ is well-defined for all $n\geq n_0$ and $j\in\{0,\ldots,b^n\}$.
With the choice of $\Theta(x):=x^p$, $x\in[0,1)$, where $p\geq 1$, the continuous $\Theta$-variation function of $f$ in Definition \ref{Def_cont_Theta_variation} 
 coincides with the usual notion of a continuous $p^\mathrm{th}$-variation function of $f$ along the sequence of $b$-adic partitions $\Pi_n$, $n\in\NN$,
 see, e.g., Cont and Perkowski \cite[Definition 1.1 and Lemma 1.3]{ConPer}.

Throughout this paper, we will use the continuous function $\Theta:[0,1)\to\RR_+$ given by
 \begin{align}\label{def_Theta}
   \Theta(x):=\begin{cases}
              \frac{x}{\sqrt{-\ln(x)}} & \text{if $x\in(0,1)$,}\\
              0 & \text{if $x=0$,}
            \end{cases}
 \end{align}
that has already been considered in Han et al.\ \cite{HanSchZha}. 
Note that $\Theta$ is strictly increasing and strictly convex, since $\Theta'(x)>0$ and 
 $\Theta''(x)>0$ for all $x\in(0,1)$, as can be easily checked.

According to Folland \cite[page 32]{Fol}, 
 we say that a function $\phi:[0,1]\to\RR$ is piecewise continuous if there exist at most finitely many points $0\leq t_1<\cdots<t_k\leq1$ 
 for some $k\in\ZZ_+$ such that $\phi$ is continuous except at the points $t_1,\ldots,t_k$ 
 and the left- and right-hand limits $\phi(t_{\ell}-)$ and $\phi(t_\ell+)$ exist in $\RR$ for all $\ell\in\{1,\ldots,k\}$. 
In case $t_{1}=0$ or $t_{k}=1$, we only require the existence of $\phi(t_{1}+)$ or $\phi(t_{k}-)$ in $\RR$, respectively.
Similarly, we say that a function $\phi:[0,1]\to\RR$ is piecewise continuously differentiable if there exist 
 at most finitely many points $0< t_1<\cdots<t_k<1$ for some $k\in\ZZ_+$ such that $\phi$ is continuously differentiable 
 except at the points $t_1,\ldots,t_k$, and the left- and right-hand limits $\phi'(t_{\ell}-)$ and $\phi'(t_\ell+)$, $\ell\in\{1,\ldots,k\}$, 
 and, additionally, $\phi'(0+)$ and $\phi'(1-)$  exist in $\RR$.

In this paper, we will study the existence of a continuous $\Theta$-variation function 
 of a function $f:[0,1]\to\RR$ along the sequence of $b$-adic partitions $\Pi_{n}=\{jb^{-n}: j=0,1,\ldots,b^n\}$, $n\in\NN$, where $b\in\NN\setminus\{1\}$ and $f$ is given by
 \begin{align}\label{help7}
   f(t):=\sum_{m=0}^\infty b^{-m}\phi(b^m t),\qquad t\in[0,1],
 \end{align}
 where we assume that $\phi:\RR\to\RR$ fulfills the following conditions:\\[-9mm]
  \begin{enumerate}\itemsep=-0.5mm
    \item[(H1)] $\phi$ is periodic with period $1$ and vanishes on $\ZZ$,
    \item[(H2)] $\phi$ is continuous and its restriction to $[0,1]$ is piecewise continuously differentiable,
    \item[(H3)] the $(jb)^{\mathrm{th}}$-Fourier coefficient of $\phi$ is zero for all $j\in \ZZ\setminus\{0\}$.
  \end{enumerate}

The function $f$ defined in \eqref{help7} is a lacunary series (function) that fulfills the so-called Hadamard gap condition, 
 since $b^{m+1}/b^m= b>1$ for each $m\in\ZZ_+$.
For details on lacunary series, we refer to Aistleitner et al.\ \cite{AisBerTic}.
Note that we will also study some properties of lacunary trigonometric functions in Lemma \ref{lacunaryHoelder} and Example \ref{Ex_lacunary}.
Lacunary sums can be considered as random variables on the probability space 
 $[0,1]$ equipped with the Borel sets and Lebesgue measure, and they often mimic the behaviour 
 of sums of independent and identically distributed random variables provided that the Hadamard gap condition holds, see Aistleitner et al.\ \cite[Section 7]{AisBerTic}.
This property of lacunary sums is usually called ''almost independent'' behaviour in the literature.
Recently, Aistleitner et al.\ \cite{AisGanKabProRam} have studied large deviation principles for lacunary sums that satisfy the Hadamard gap condition
 or some other gap condition.
Very recently, Formica et al.\ \cite{ForOstSir} have investigated precise exponential tail estimates for trigonometric
 lacunary series.
 
We also mention that the function $f$ defined in \eqref{help7} is called a special Weierstrass-type function as well 
 (see, e.g., the Introduction of Barczy and Kern \cite{BarKer}).

\begin{Rem}\label{Rem_hypotheses}
(i) The hypotheses (H1) and (H2) imply that $\phi$ is Lipschitz continuous, see Lemmas \ref{Lemma_Lipcont_1} and \ref{Lemma_Lipcont_2}.

(ii). We call the attention that the parameter $b\in\NN\setminus\{1\}$ in the definition of $f$ in \eqref{help7} and in hypothesis (H3) 
 is the same as the parameter $b$ of the sequence of $b$-adic partitions along which the existence of a continuous $\Theta$-variation
 function of $f$ will be investigated.

(iii).  
According to Theorem 2.5 in Folland \cite{Fol}, by the hypotheses (H1) and (H2), 
 the Fourier series of $\phi$ converges to $\phi$ absolutely and uniformly on $\RR$. 
Let $(c_j)_{j\in\ZZ}$ denote the Fourier coefficients of $\phi$
 (see \eqref{Fourier_coefficients}).
Note also that the hypotheses (H1) and (H2) together with Theorem 2.6 in Folland \cite{Fol} 
 implies that $\sum_{j\in\ZZ}j^{2}|c_j|^{2}<\infty$. 
Moreover, if hypothesis (H1) and $\sum_{j\in\ZZ}|j\,c_j|<\infty$ hold, then the proof of the second part of Theorem 2.6 
 in Folland \cite{Fol} shows that $\phi$ is continuously differentiable and thus fulfills (H2).
\proofend
\end{Rem} 

Let us also introduce the following hypothesis:
  \begin{itemize}
   \item[(H2$^{\ast}$)] $\phi$ is H\"older continuous with exponent $\beta\in(0,1]$ such that it has an absolutely convergent Fourier series,
         i.e., $\sum_{j\in\ZZ} \vert c_j\vert<\infty$, where $c_j$, $j\in\ZZ$, denote the Fourier coefficients of $\phi$
         (see \eqref{Fourier_coefficients}).
  \end{itemize}

In some cases, we will also consider the function $f$ defined in \eqref{help7} when the hypothesis (H2) is replaced by 
 (H2$^{\ast}$), which, in the presence of (H1), is weaker than (H2) (see part (i) of Remark \ref{Rem_hypotheses_2}).

\begin{Rem}\label{Rem_hypotheses_2}
(i). If the hypotheses (H1) and (H2$^{\ast}$) hold, then the Fourier series of $\phi$ converges to $\phi$ absolutely and uniformly on $\RR$,
 see, e.g., Katznelson \cite[pages 31 and 32]{Kat}.
Further, if the hypotheses (H1) and (H2) hold, then the hypothesis (H2$^*$) with $\beta=1$ holds as well (following from Lemmas \ref{Lemma_Lipcont_1} and \ref{Lemma_Lipcont_2}).

(ii).
If the hypothesis (H1) holds and $\phi$
 is H\"older continuous with exponent $\beta\in(\frac{1}{2},1]$, then $\phi$ has an absolutely
 convergent Fourier series due to Bernstein's theorem, see, e.g., Katznelson \cite[Theorem 6.3 on page 33]{Kat} or 
 Stein and Shakarchi \cite[part (d) of Exercise 16 on page 93]{SteSha}.
Further, according to Zygmund's theorem (see, e.g., Katznelson \cite[Theorem 6.4 on page 33]{Kat}),
 if the hypothesis (H1) holds, $\phi$ is H\"older continuous with exponent $\beta\in(0,1]$
 and is of bounded variation, then $\phi$ has an absolutely convergent Fourier series.
\proofend
\end{Rem}

Next, we will recall a probabilistic representation of $V^{\Theta,1}_n(f)$ (due to Han et al.\ \cite[formula (2.2)]{HanSchZha})
 in terms of the expectation of some appropriate random variable.
In order to do so, first, we introduce some notations.
For each $m\in\NN$ and $j\in\ZZ_+$, let
 \begin{align*}
  \lambda_{m,j}:=\frac{\phi((j+1)b^{-m}) - \phi(jb^{-m})}{b^{-m}},
 \end{align*}
 which is the slope of the line connecting $(jb^{-m},\phi(jb^{-m}))$ with $((j+1)b^{-m},\phi((j+1)b^{-m}))$.
Let $(U_n)_{n\in\NN}$ be a sequence of independent and identically distributed random variables such that
 $U_1$ is uniformly distributed on the finite set $\{0,1,\ldots,b-1\}$.
Further, let $R_0:=0$, and, for each $m\in\NN$, let us define the random variables 
 \begin{align}\label{help33}
  R_m:=\sum_{i=1}^m U_i b^{i-1} \qquad \text{and}\qquad Y_m:=\lambda_{m,R_m}.
 \end{align}
One can easily check that $R_m$ is uniformly distributed on the set $\{0,\ldots,b^m-1\}$ for each $m\in\ZZ_+$.

If $\phi:\RR\to\RR$ is H\"older continuous with exponent $\beta\in(0,1]$, then there exists a constant $C\in\RR_{++}$ such that
 \begin{align}\label{phi_Holder}
    \vert \phi(x) - \phi(y)\vert \leq C \vert x - y\vert^\beta, \qquad x,y\in\RR,
 \end{align}
and consequently we have
 \begin{align}\label{Y_estimate}
  \vert Y_m\vert\leq C{b^{m(1-\beta)}}, \qquad  m\in\NN.
 \end{align}
If the hypotheses (H1) and (H2) are fulfilled, then, by Lemmas \ref{Lemma_Lipcont_1} and \ref{Lemma_Lipcont_2},
 we have that $\phi$ is Lipschitz continuous, i.e., $\beta=1$, and hence $|Y_{m}|\leq C$ for all $m\in\NN$, i.e., the sequence $(Y_m)_{m\in\NN}$ is bounded.

We now provide an expression for $V^{\Theta,1}_n(f)$, $n\in\NN$.
The result itself is not entirely new, it is due to Han et al.\ \cite[formula (2.2)]{HanSchZha},
 but we present it under a weaker hypothesis (supposing the H\"older continuouity of $\phi$ instead of Lipschitz continuity)
 and clarify the well-definedness of the expression in question for sufficiently large \ $n\in\NN$ as well.

\begin{Lem}\label{Lem_p_var_expression}
Let $b\in\NN\setminus\{1\}$ and $\phi:\RR\to\RR$ be a function satisfying the hypothesis (H1) and suppose that $\phi$ is H\"older continuous with exponent $\beta\in(0,1]$.
Then the function $f$ given by \eqref{help7} is well-defined and continuous.
Further, we have that
 \begin{align}\label{help15_Theta}
   V^{\Theta,1}_n(f) = b^n \EE\left( \Theta\left( b^{-n}\left\vert \sum_{m=1}^n Y_m \right\vert \right) \right)
                        \qquad \text{for sufficiently large \ $n\in\NN$},
 \end{align}
 where the functions $\Theta:[0,1)\to\RR_+$ and $f:[0,1]\to\RR$ are given by \eqref{def_Theta} and \eqref{help7}, respectively.
\end{Lem}

The proof of Lemma \ref{Lem_p_var_expression} can be found in Appendix \ref{App_C}.

Next, we recall some recent results of Han et al.\ \cite{HanSchZha} on the $\Theta$-variation function of the Takagi function and a Weierstrass function 
 that motivated our study.

\begin{Thm}[Han et al.\ {\cite[Theorems 1.2 and 1.3]{HanSchZha}}]\label{Thm_HSZ_Thms_12_13}
Let $b\in\NN\setminus\{1\}$, $\Theta:[0,1)\to\RR_+$ be given by \eqref{def_Theta}, 
 and let $f:[0,1]\to\RR$ be given by \eqref{help7} with
\begin{itemize}
 \item[(i)] $\phi(t):=\min_{z\in\ZZ}\vert z-t\vert$, $t\in\RR$. (In this case $f$ is called the Takagi function.) Then
      \[
        \langle f \rangle^{\Theta}(t)
           = \begin{cases}
               t\cdot \sqrt{\frac{2}{\pi \ln(b)}}  &  \text{if $b$ is even,}\\[1mm]
               t\cdot \sqrt{\frac{2}{\pi\ln(b)}\cdot \frac{b+1}{b-1} }  &  \text{if $b$ is odd,}
            \end{cases}\qquad t\in[0,1].
      \]
  \item[(ii)] $\phi(t):=\nu \sin(2\pi t) + \varrho \cos(2\pi t) - \varrho$, $t\in\RR$, where $\nu,\varrho\in\RR$.
      (In this case $f$ is called a Weierstrass function.)
       Then 
       \[
        \langle f \rangle^{\Theta}(t)
         = t\cdot 2\sqrt{\frac{\pi(\nu^2+\varrho^2)}{\ln(b)}},
         \qquad t\in[0,1].
       \]
\end{itemize}
\end{Thm}

In Examples \ref{Ex_Takagi} and \ref{Ex_Weierstrass}, we check that  
 the Takagi function for an even $b$ and the Weiserstrass function considered in Theorem \ref{Thm_HSZ_Thms_12_13} 
 are special cases of the function $f$ defined in \eqref{help7} satisfying the hypotheses (H1), (H2) and (H3).
Our aim is to extend the results of Han et al.\ \cite{HanSchZha} to a function $f$ defined in \eqref{help7},
 where $\phi$ satisfies the hypotheses (H1), (H2) and (H3).

The paper is structured as follows.
Section \ref{Sec_Results}
is devoted to investigate the existence of a continuous $\Theta$-variation function of $f$ given by \eqref{help7},
where $\phi$ satisfies the hypotheses (H1), (H2) and (H3), and to derive a formula for it.
First, under the hypotheses (H1), (H2$^*$) and (H3), we introduce a square-integrable martingale 
related to the function $f$, by considering partial sums of the appropriately scaled random variables $Y_m$, $m\in\NN$
(given in \eqref{help33}), namely, $Z_n:=\sum_{m=1}^n b^{m(\beta-1)}Y_m$, $n\in\NN$,
see Proposition \ref{Pro_Z_martingale}.
Then, under the hypotheses (H1), (H2) and (H3), we prove that 
 its quadratic variation $\langle Z\rangle_n$ multiplied by $1/n$ converges $\PP$-almost surely to 
$\int_0^1 (\phi'(x))^2\,\dd x$ as $n\to\infty$,
 see Lemma \ref{Lem_quadratic_Z_almost}.
It turns out that $\frac{1}{n}\EE(\langle Z\rangle_n)$ also converges to 
$\int_0^1 (\phi'(x))^2\,\dd x$ as $n\to\infty$, see Lemma \ref{Lem_quadratic_Z_exp}.
Theorem \ref{Thm_main} contains our main result, in which, under the hypotheses (H1), (H2) and (H3),
we describe the continuous $\Theta$-variation function of $f$ along the sequence of $b$-adic partitions.
Section \ref{Section_examples} is devoted to some examples. 
In Example \ref{Ex_Takagi}, we apply Theorem \ref{Thm_main} to the Takagi function by recovering a recent result of Han et al.\ \cite[Theorem 1.2]{HanSchZha} 
 (see also part (i) of Theorem \ref{Thm_HSZ_Thms_12_13} for an even $b$).
In Example \ref{Ex_Weierstrass}, we also apply Theorem \ref{Thm_main} to Weierstrass functions by recovering the corresponding result of
 Han et al.\ \cite[Theorem 1.2]{HanSchZha} (see also part (ii) of Theorem \ref{Thm_HSZ_Thms_12_13}).
We emphasize that, in case of the Takagi function, our method of proof is different from the one in Han et al.\ \cite[Theorem 1.2]{HanSchZha}.
As a new example, we apply Theorem \ref{Thm_main} to lacunary trigonometric functions, see Example \ref{Ex_lacunary}.
We close the paper with three appendices which contain some auxiliary results.
In Appendix \ref{App_A}, we recall a result on the comparison of two canonical representations 
of real numbers in $[0,1)$. 
Appendix \ref{App_B} is devoted to some auxiliary results on piecewise continuous and piecewise continuously differentiable functions.
For example, in Lemma \ref{Lemma_Lipcont_1}, we prove that
a continuous and piecewise continuously differentiable function defined on $[0,1]$,
is Lipschitz continuous as well. 
We believe that the results in Appendices \ref{App_A} and \ref{App_B} are well-known in the literature, but we cannot address any reference for it,
 and hence, for completeness, we provide proofs as well.
In Appendix \ref{App_C}, we collect some (auxiliary) results that can partially be found in Han et al.\ \cite{HanSchZha} 
 and are extended to our needs.
Namely, we give an extension of formula (2.3) in the proof of Lemma 2.1 in Han et al.\ \cite{HanSchZha},
see Lemma \ref{Lem_HSZ_aux}.
In Lemma \ref{Lem_HSZ}, we give a detailed proof of Lemma 2.1 in Han et al.\ \cite{HanSchZha},
 because its steps are needed in the proof of Theorem \ref{Thm_main}.
We also give a generalization of Lemma 2.3 in Han et al.\ \cite{HanSchZha} 
to the case of a piecewise continuous function, see Lemma \ref{Lem_weak_conv}.

\section{Continuous $\Theta$-variation function of $f$ given by \eqref{help7}}\label{Sec_Results}

Our aim is to investigate the existence of a continuous $\Theta$-variation function of $f$ given by \eqref{help7} 
 along the sequence of $b$-adic partition $\Pi_{n}$, $n\in\NN$, where $\phi$ satisfies the hypotheses (H1), (H2) and (H3) and $\Theta$ is 
 defined by \eqref{def_Theta}.
Note that Lemma \ref{Lem_p_var_expression} implies that such a function $f$ is indeed well-defined and continuous,
 since the hypotheses (H1), (H2) and Lemmas \ref{Lemma_Lipcont_1} and \ref{Lemma_Lipcont_2} yield that $\phi$ is Lipschitz continuous 
 and hence one can apply Lemma \ref{Lem_p_var_expression} with $\beta=1$.

The next result is a generalization of Lemma 2.2 in Han et al.\  \cite{HanSchZha}
 in the sense that instead of the particular Lipschitz continuous function 
 $\phi$ given in part (ii) of Theorem \ref{Thm_HSZ_Thms_12_13} (corresponding to Weierstrass functions),
 we consider a H\"older continuous function $\phi$ satisfying the hypotheses (H1), (H2*) and (H3).

Recall the definition of $Y_{m}$, $m\in\NN$, in \eqref{help33} incorporating the independent and identically distributed random variables 
 $U_{1},\ldots,U_{m}$, $m\in\NN$, with a uniform distribution on $\{0,\ldots,b-1\}$.
  
\begin{Pro}\label{Pro_Z_martingale}
Let us consider the function $f$ defined by \eqref{help7} such that the hypotheses (H1), (H2$^*$) and (H3) hold.
Then the stochastic process $(Z_n)_{n\in\ZZ_+}$ given by $Z_0:=0$ and
 \begin{align}\label{Def_Z}
    Z_n:=\sum_{m=1}^n b^{m(\beta-1)}Y_m, \qquad n\in\NN,
 \end{align}
 is a zero mean, square-integrable martingale with respect to the filtration $(\cF_n)_{n\in\ZZ_+}$,
 where $\cF_n:=\sigma(U_1,\ldots,U_n)$, $n\in\NN$, and $\cF_0:=\{\emptyset,\Omega\}$. 
\end{Pro}

\begin{proof}
Note that $\cF_n=\sigma(R_n)$, $n\in\ZZ_{+}$, since $R_{0}=0$ and $R_n = \sum_{i=1}^n U_i b^{i-1}$ is a canonical $b$-adic expansion for all $n\in\NN$,
 and hence there is a one-to-one correspondence between $R_n$ and $U_i$, $i\in\{1,\ldots,n\}$ for all $n\in\NN$,
 see, e.g., Vakil \cite[Theorem B.1.2]{Vak}.
Therefore, taking into account the definition \eqref{help33} of $Y_m$, $m\in\ZZ_+$, as well,  
 $Z_n$ is measurable with respect to $\cF_n$ for all $n\in\ZZ_{+}$.
Using \eqref{Y_estimate}, we have $\vert Z_n\vert\leq C n$, $n\in\ZZ_{+}$, yielding that $Z_n$ is square-integrable for all $n\in\ZZ_{+}$.
Since, for all $n\in\NN$ ,we have
 \begin{align*}
   \EE(Z_n\mid \cF_{n-1}) & = \EE(Z_{n-1}\mid \cF_{n-1}) + \EE( b^{n(\beta-1)}Y_n\mid \cF_{n-1}) 
                           = Z_{n-1} + b^{n(\beta-1)} \EE(Y_n\mid \cF_{n-1}),
 \end{align*}
 in order to show the martingale property in question it is enough to verify that
 $\EE(Y_n\mid R_{n-1})=0$ for all $n\in\NN$. 
It also implies that $\EE(Y_n)=0$, $n\in\NN$, and hence $\EE(Z_n)=0$, $n\in\ZZ_+$.
 By hypothesis (H2$^{\ast}$), $\phi$ has an absolutely convergent Fourier series 
 \begin{align}\label{phi_Fourier_series}
   \phi(t) = \sum_{j\in\ZZ} c_j \exp(2\pi \ii j t), \qquad  t\in\RR,
 \end{align}
 where $\sum_{j\in\ZZ} \vert c_j\vert <\infty$ holds for the Fourier coefficients
 \begin{align}\label{Fourier_coefficients}
   c_j:=\int_0^1 \phi(t) \exp(-2\pi \ii j t)\,\dd t, \qquad j\in\ZZ.
 \end{align} 
Due to hypothesis (H3), we have $c_j=0$ for all $j\in b\ZZ\setminus\{0\}$.
By \eqref{help33}, we have for all $n\in\NN$
 \begin{equation}\label{help_Fourier_0_5}\begin{split}
   Y_n &= b^n \big( \phi((R_n+1)b^{-n}) - \phi(R_nb^{-n}) \big) \\
       &= b^n \sum_{j\in\ZZ} c_j \Big( \exp(2\pi \ii j (R_n+1)b^{-n}) - \exp(2\pi \ii j R_n b^{-n}) \Big)\\
       &= b^n \sum_{j\notin b\ZZ} c_j \Big( \exp(2\pi \ii j (R_n+1)b^{-n}) - \exp(2\pi \ii j R_n b^{-n}) \Big),
 \end{split}\end{equation}
 where, at the last equality, we also used that, for $j=0$, the term $\exp(2\pi \ii j (R_n+1)b^{-n}) - \exp(2\pi \ii j R_n b^{-n})$ is equal to $1-1=0$.
In what follows, we will use that $R_n = R_{n-1}+U_n b^{n-1}$, $n\in\NN$, where $R_{n-1}$ and $U_n$ are independent,
 $U_n$ is uniformly distributed on $\{0,1,\ldots,b-1\}$ and $R_{n-1}$ is uniformly distributed on $\{0,1,\ldots,b^{n-1}-1\}$.
Using \eqref{help_Fourier_0_5} and $\sum_{j\in\ZZ} \vert c_j\vert <\infty$, for each $r\in\{0,1,\ldots,b^{n-1}-1\}$ we obtain that
 \begin{align*}
 \EE(Y_n\mid R_{n-1}=r)
    & = \EE\left( b^n \sum_{j\notin b\ZZ} c_j \Big( \exp(2\pi \ii j (R_n+1)b^{-n}) - \exp(2\pi \ii j R_n b^{-n}) \Big) \,\bigg\vert\,  R_{n-1}=r \right) \\
    & = b^n \EE\left( \sum_{j\notin b\ZZ} c_j 
                   \Big( \exp(2\pi \ii j (r+ U_n b^{n-1}+1)b^{-n}) - \exp(2\pi \ii j (r+ U_n b^{n-1}) b^{-n}) \Big)
                \right)\\
    & = b^{n-1} \sum_{\ell=0}^{b-1} \sum_{j\notin b\ZZ}c_j
          \Big( \exp(2\pi \ii j (r+\ell b^{n-1}+1)b^{-n}) - \exp(2\pi \ii j (r+\ell b^{n-1}) b^{-n}) \Big)\\
    & = b^{n-1} \sum_{j\notin b\ZZ}c_j \sum_{\ell=0}^{b-1}
        \Bigg[ \exp\Big(2\pi \ii j \Big( (r+1) b^{-n} + \frac{\ell}{b}\Big)\Big) - \exp\Big(2\pi \ii j \Big(r b^{-n} + \frac{\ell}{b}\Big)\Big) \Bigg].
 \end{align*}
Using again that $\sum_{j\in\ZZ} \vert c_j\vert <\infty$, for all $x\in\RR$, we have that
 \begin{align}\label{help_Fourier2}
  \begin{split}
 \sum_{j\notin b\ZZ}c_j \sum_{\ell=0}^{b-1} \exp\Big(2\pi \ii j \Big( x + \frac{\ell}{b}\Big) \Big)
    & =  \sum_{j\notin b\ZZ} c_j \exp(2\pi \ii j x) \sum_{\ell=0}^{b-1} \exp\Big(2\pi \ii j \frac{\ell}{b} \Big) \\
    & = \sum_{j\notin b\ZZ} c_j \exp(2\pi \ii j x)\, \frac{\exp(2\pi \ii j) - 1}{\exp(2\pi \ii \frac{j}{b}) - 1}=0,
 \end{split}
 \end{align}
 where at the second equality we used that $\exp(2\pi \ii \frac{j}{b})\ne 1$ due to $j\notin b\ZZ$.
Applying \eqref{help_Fourier2} with $x=(r+1) b^{-n}$ and $x=r b^{-n}$, respectively, we get
 \[
 \EE(Y_n\mid R_{n-1}=r)=0,\qquad r\in\{0,1,\ldots,b^{n-1}-1\},
 \]
 concluding the proof.
\end{proof}

Let $(\langle Z\rangle_n)_{n\in\ZZ_+}$ denote the quadratic variation process of 
 the square-integrable martingale $(Z_n)_{n\in\ZZ_+}$ introduced in \eqref{Def_Z} with $\beta=1$, 
  i.e., $\langle Z\rangle_0:=0$ and
 \begin{align}\label{Def_Z_quadratic_variation}
   \langle Z\rangle_n:=\sum_{m=1}^n \EE( (Z_m - Z_{m-1})^2 \mid \cF_{m-1})
                      = \sum_{m=1}^n \EE( Y_m^2 \mid \cF_{m-1}),\qquad n\in\NN,
 \end{align}
 where $\cF_n = \sigma(U_1,\ldots,U_n)=\sigma(R_{n})$, $n\in\NN$, and $\cF_0=\{\emptyset,\Omega\}$.
 
\begin{Lem}\label{Lem_quadratic_Z_almost}
Let us consider the function $f$ defined by \eqref{help7} such that $\phi$ fulfills the hypotheses (H1), (H2) and (H3).
Then, for the quadratic variation process $(\langle Z\rangle_n)_{n\in\ZZ_+}$ defined in \eqref{Def_Z_quadratic_variation},
 we have
 \[
   \frac{1}{n}\langle Z\rangle_n \to  \int_0^1 (\phi'(x))^2\,\dd x\in\RR_+ \quad \text{as $n\to\infty$ \ $\PP$-almost surely.}
 \] 
\end{Lem}

\begin{Rem}\label{phizero}
(i). Note that the restriction of $\phi$ onto $[0,1]$ is piecewise continuously differentiable due to hypothesis (H2),
 and thus $\phi'$ exists on $(0,1)$ except for at most finitely many exceptional points $0<t_{1}<\cdots<t_{k}<1$ for some $k\in\ZZ_{+}$,
and $\phi'(0+)$ and $\phi'(1-)$ exist as well.
We can extend the domain $[0,1]\setminus\{t_i : i=1,\ldots,k\}$ of $\phi'$ by adding arbitrary values at the exceptional points in question. 
This extension is Riemann-integrable on $[0,1]$, since it is continuous Lebesgue almost everywhere. 
Further, the Lebesgue integral of the square of the extension in question on $[0,1]$
 does not depend on the values of the extension at the possible exceptional points. 
On the probability space $[0,1]$ equipped with the Borel $\sigma$-algebra and the Lebesgue measure restricted to $[0,1]$, 
 every extension defines the same element in $L^{2}([0,1])$, so that we keep the same notion $\phi'$ also for the extension.

(ii). If the hypotheses (H1) and (H2) hold and $\phi$ is not the identically zero function, 
 then $\int_{0}^{1}(\phi'(x))^{2}\,\dd x>0$. 
Indeed, on the contrary, assume that $\int_0^1 (\phi'(x))^2 \,\dd x=0$.
Since $\phi'$ is piecewise continuous, we get $\phi'(x)=0$ for all $x\in(t_{j},t_{j+1})$ with $j\in\{0,\ldots,k\}$, where $t_{0}:=0$ and $t_{k+1}:=1$. 
Hence for all $j\in\{0,\ldots,k\}$, there exists a constant $C_{j}\in\RR$ such that  $\phi(x)=C_{j}$ for all $x\in(t_{j},t_{j+1})$. 
By the hypotheses (H1) and (H2), $\phi$ is continuous and $\phi(0)=0=\phi(1)$, which implies that $\phi(x)=0$ for all $x\in[0,1]$.
\proofend
\end{Rem}

\noindent{\bf Proof of Lemma \ref{Lem_quadratic_Z_almost}.}
By hypothesis (H1) and Lemma \ref{Lemma_piececont_bound_meas}, 
 for any fixed $\ell\in\{0,\ldots,b-1\}$, the function $[0,1]\ni x\mapsto g(x):=(\phi'(\frac{x+\ell}{b}))^{2}$ is bounded 
 and piecewise continuous (where $\phi'$ is meant in the sense of part (i) of Remark \ref{phizero}).
Therefore, by the boundedness of $\phi'$ and Lemma \ref{Lem_weak_conv}, we get as $n\to\infty$
\begin{equation}\label{as-R-int-mean-value}\begin{split}
 \frac{1}{n}\sum_{m=1}^{n}
          \Bigg(\phi'\left(\frac{R_{m-1}\,b^{-(m-1)}+\ell}{b}\right)\Bigg)^2\to\int_{0}^{1}\Bigg(\phi'\left(\frac{x+\ell}{b}\right)\Bigg)^2\,\dd x \qquad \text{$\PP$-almost surely.}
\end{split}\end{equation}
Recall that, under the hypotheses (H1) and (H2), the Fourier series of $\phi$ converges to $\phi$ absolutely and uniformly on $\RR$,
 see part (iii) of Remark \ref{Rem_hypotheses}.
Using \eqref{help_Fourier_0_5} and  that $R_n = R_{n-1}+U_n b^{n-1}$, $n\in\NN$, where $R_{n-1}$ and $U_n$ are independent,
 $U_n$ is uniformly distributed on $\{0,1,\ldots,b-1\}$ and $R_{n-1}$ is uniformly distributed on $\{0,1,\ldots,b^{n-1}-1\}$,
 similarly as in the proof of Proposition \ref{Pro_Z_martingale},
 for each $m\in\NN\setminus\{1\}$ and $r\in\{0,1,\ldots,b^{m-1}-1\}$, we have that
 \begin{align*}
 &\EE(Y_m^2\mid R_{m-1}=r)\\
 &\quad= \EE\left( \left( b^m \sum_{j\notin b\ZZ} c_j \Big( \exp(2\pi \ii j (R_m+1)b^{-m}) - \exp(2\pi \ii j R_m b^{-m}) \Big) \right)^2 \,\Bigg\vert\, R_{m-1}=r \right) \\
 &\quad= b^{2m-1} \sum_{\ell=0}^{b-1} \left(\sum_{j\notin b\ZZ}c_j
           \Big( \exp(2\pi \ii j (r+\ell b^{m-1}+1)b^{-m}) - \exp(2\pi \ii j (r+\ell b^{m-1}) b^{-m}) \Big) \right)^2\\
 &\quad= b^{2m-1} \sum_{\ell=0}^{b-1} \left(\sum_{j\notin b\ZZ}c_j
           \Bigg[ \exp\Big(2\pi \ii j \Big( (r+1)b^{-m}+ \frac{\ell}{b}\Big)\Big) - \exp\Big(2\pi \ii j \Big(r b^{-m} + \frac{\ell}{b}\Big)\Big) \Bigg] \right)^2.
 \end{align*}
Hence, using that $\cF_n = \sigma(R_n)$, $n\in\ZZ_+$ (explained at the beginning of the proof of Proposition \ref{Pro_Z_martingale}),
 by \eqref{phi_Fourier_series} and hypothesis (H3), we get that
\begin{align}\label{help_quad_1}
 \begin{split}
  \frac{1}{n}\langle Z\rangle_n 
  &= \frac{1}{n}\sum_{m=1}^n b^{2m-1}
        \sum_{\ell=0}^{b-1} \Bigg(\sum_{j\notin b\ZZ}c_j
                \bigg[ \exp\Big(2\pi \ii j \Big( (R_{m-1}+1)b^{-m}+ \frac{\ell}{b}\Big)\Big) \\
  &\phantom{= \frac{1}{n}\sum_{m=2}^n b^{2m-1} \sum_{\ell=0}^{b-1} \Bigg(\sum_{j\notin b\ZZ}c_j\bigg[\,}               
                     - \exp\Big(2\pi \ii j \Big(R_{m-1} b^{-m} + \frac{\ell}{b}\Big)\Big) \bigg] \Bigg)^2\\
  & = \frac{1}{b}\sum_{\ell=0}^{b-1} \frac{1}{n} \sum_{m=1}^n
          \Bigg(\frac{\phi(\frac{(R_{m-1}+1)b^{-(m-1)}+\ell}{b})-\phi(\frac{R_{m-1}\,b^{-(m-1)}+\ell}{b})}{b^{-m}}\Bigg)^2.
 \end{split}         
 \end{align} 
For each $\ell\in\{0,\ldots,b-1\}$ and $m\in\NN$, define the random intervals
 $$
 I_{m,\ell}=\left[\frac{R_{m-1}b^{-(m-1)}+\ell}{b},\frac{(R_{m-1}+1)b^{-(m-1)}+\ell}{b}\right],
 $$
 and let $t_{1},\ldots,t_{k}$ (where $k\in\ZZ_+$) be the at most finitely many exceptional points, where the restriction of $\phi$ onto $(0,1)$ 
 is not continuously differentiable. 
Let $t_0:=0$ and $t_{k+1}:=1$.
Using Lemmas \ref{Lemma_Lipcont_1} and \ref{Lemma_Lipcont_2}, hypotheses (H1) and (H2) yield that $\phi$ is Lipschitz continuous.
To ensure 
\begin{equation}\label{BClimsup*}
\PP\left(\limsup_{m\to\infty}\{t_j\in I_{m,\ell}\}\right)=0\quad\text{ for all }j\in\{0,1,\ldots,k,k+1\}\text{ and }\ell\in\{0,\ldots,b-1\},
\end{equation}
by the Borel-Cantelli lemma, it suffices to show that
\begin{equation}\label{BClimsup*suff}
\sum_{m=1}^{\infty}\PP\left(t_j\in I_{m,\ell}\right)<\infty\quad\text{ for all }j\in\{0,1,\ldots,k,k+1\}\text{ and }\ell\in\{0,\ldots,b-1\}.
\end{equation}
Since $R_{m-1}$ is uniformly distributed on $\{0,\ldots,b^{m-2}-1\}$ for all $m\geq 2$, $m\in\NN$, we get
\begin{align*}
\PP\left(t_j\in I_{m,\ell}\right) & =\PP\left(\frac{R_{m-1}b^{-(m-1)}+\ell}{b}\leq t_{j}\leq\frac{(R_{m-1}+1)b^{-(m-1)}+\ell}{b}\right)\\
& =\PP\left( b^m t_{j}-b^{m-1}\ell-1 \leq R_{m-1}\leq b^mt_j-b^{m-1}\ell\right)\leq\frac2{b^{m-2}},
\end{align*}
which proves \eqref{BClimsup*suff}. Let $J_{\ell}$ be the set consisting of those $m\in\NN$ such that $I_{m,\ell}$ contains 
 at least one of the points $t_0,\ldots, t_{k+1}$.
Similarly, as in the proof of Lemma \ref{Lem_weak_conv}, we can conclude from \eqref{BClimsup*} that $J_{\ell}$ is finite $\PP$-almost surely.
From the Lipschitz continuity of $\phi$ (of which the Lipschitz constant is denoted by $C$ according to \eqref{phi_Holder}),
 we can deduce that, for all $\ell\in\{0,\ldots, b-1\}$, we have that $\PP$-almost surely as $n\to\infty$
\begin{equation}\label{except-nr-1}
0\leq \frac{1}{n}\sum_{m\in\{1,\ldots,n\}\cap J_{\ell}}
          \Bigg(\frac{\phi(\frac{(R_{m-1}+1)b^{-(m-1)}+\ell}{b})-\phi(\frac{R_{m-1}\,b^{-(m-1)}+\ell}{b})}{b^{-m}}\Bigg)^2
          \leq\frac{1}{n}\,|J_{\ell}|\,C^2\to0.
\end{equation}
By hypothesis (H2) and Lemma \ref{Lemma_piececont_bound_meas}, we have that $\phi'$ is bounded on $[0,1]$,
 and hence there exists  $\widetilde C\in\RR_{++}$ such that  
 $(\phi'(x))^{2}\leq\widetilde C$ for all $x\in[0,1]$, yielding that, for all $\ell\in\{0,\ldots,b-1\}$, $\PP$-almost surely
 it holds that  
\begin{equation}\label{except-nr-2}
0\leq \frac{1}{n}\sum_{m\in\{1,\ldots,n\}\cap J_{\ell}} \left(\phi'\left(\frac{R_{m-1}b^{-(m-1)}+\ell}{b}\right)\right)^2\leq \frac{1}{n}\,|J_{\ell}|\,\widetilde C\to0
 \qquad \text{as $n\to\infty$,}
\end{equation}
regardless of the values of $\phi'$ at the possible exceptional points $t_{1},\ldots,t_{k}$ 
(see part (i) of Remark \ref{phizero}).
For all $\ell\in\{0,\ldots,b-1\}$ and $m\not\in J_{\ell}$, we know that $\phi$ is continuously differentiable on $I_{m,\ell}$ and hence, by the mean value theorem, 
 there exists a $T_{m,\ell}\in[R_{m-1},R_{m-1}+1]$ such that 
 \begin{equation}\label{mvt-random}
 \phi\left(\frac{(R_{m-1}+1)b^{-(m-1)}+\ell}{b}\right)-\phi\left(\frac{R_{m-1}\,b^{-(m-1)}+\ell}{b}\right)=b^{-m}\,\phi'\left(\frac{T_{m,\ell}\,b^{-(m-1)}+\ell}{b}\right),
 \end{equation}
where
\begin{equation}\label{mvt-random-diff}
 \left|\frac{R_{m-1}b^{-(m-1)}+\ell}{b}-\frac{T_{m,\ell}\,b^{-(m-1)}+\ell}{b}\right|=\left|R_{m-1}-T_{m,\ell}\right|b^{-m}\leq b^{-m}.
 \end{equation}
Note that we do not insist on $T_{m,\ell}$ to be a random variable, 
 but the right-hand side of \eqref{mvt-random} is a random variable, and hence its left-hand side as well.
When the mean value theorem was applied, we simply have chosen an intermediate value for each $\omega\in\Omega\setminus\{m\in J_\ell\}$.
On the compact set $[0,1]\setminus\big(\bigcup_{\ell=0}^{b-1}\bigcup_{m\in J_{\ell}}I_{m,\ell}\big)^{\circ}$, 
 where $\circ$ denotes the interior of a set, the function $(\phi')^{2}$ is uniformly continuous and hence, by \eqref{mvt-random-diff}, 
 for any $\varepsilon>0$ there exists $m(\varepsilon)\in\NN$ with $m(\varepsilon)>\max J_{\ell}$ for all $\ell\in\{0,\ldots, b-1\}$ such that 
 for all $m\geq m(\varepsilon)$ and $\ell\in\{0,\ldots,b-1\}$, we have that
\begin{equation}\label{ucon-random}
\left|\Bigg(\phi'\left(\frac{R_{m-1}\,b^{-(m-1)}+\ell}{b}\right)\Bigg)^2-\Bigg(\phi'\left(\frac{T_{m,\ell}\,b^{-(m-1)}+\ell}{b}\right)\Bigg)^2\right|<\vare
\end{equation}
 holds $\PP$-almost surely. 
Recall that, for any two real sequences $(a_n)_{n\in\NN}$ and $(b_n)_{n\in\NN}$, it holds that
 $\limsup_{n\to\infty} (a_n+b_n) \leq \limsup_{n\to\infty} a_n + \limsup_{n\to\infty} b_n$,
 whenever the right-hand side of the inequality is well-defined.
Altogether, successively using \eqref{help_quad_1}, \eqref{except-nr-1}, \eqref{mvt-random}, \eqref{ucon-random}, \eqref{except-nr-2}, and \eqref{as-R-int-mean-value} as well,
 we get $\PP$-almost surely
\begin{align*}
\limsup_{n\to\infty} \frac{1}{n}\langle Z\rangle_n
 & \leq \frac{1}{b}\sum_{\ell=0}^{b-1} \limsup_{n\to\infty}\frac{1}{n} \sum_{m=1}^n
          \Bigg(\frac{\phi(\frac{(R_{m-1}+1)b^{-(m-1)}+\ell}{b})-\phi(\frac{R_{m-1}\,b^{-(m-1)}+\ell}{b})}{b^{-m}}\Bigg)^2\\
& =\frac{1}{b}\sum_{\ell=0}^{b-1} \limsup_{n\to\infty}\frac{1}{n} \sum_{m\in\{1,\ldots,n\}\setminus J_{\ell}}
          \Bigg(\frac{\phi(\frac{(R_{m-1}+1)b^{-(m-1)}+\ell}{b})-\phi(\frac{R_{m-1}\,b^{-(m-1)}+\ell}{b})}{b^{-m}}\Bigg)^2\\ 
& =\frac{1}{b}\sum_{\ell=0}^{b-1} \limsup_{n\to\infty}\frac{1}{n} \sum_{m\in\{1,\ldots,n\}\setminus J_{\ell}}
          \Bigg(\phi'\left(\frac{T_{m,\ell}\,b^{-(m-1)}+\ell}{b}\right)\Bigg)^{2}\\
& =\frac{1}{b}\sum_{\ell=0}^{b-1} \limsup_{n\to\infty}\frac{1}{n} \sum_{m\in\{m(\vare),\ldots,n\}}
          \Bigg(\phi'\left(\frac{T_{m,\ell}\,b^{-(m-1)}+\ell}{b}\right)\Bigg)^{2}\\          
& \leq\frac{1}{b}\sum_{\ell=0}^{b-1}\limsup_{n\to\infty}\frac{1}{n} \sum_{m\in\{m(\vare),\ldots,n\}\setminus J_{\ell}}
          \Bigg(\Bigg(\phi'\left(\frac{R_{m-1}\,b^{-(m-1)}+\ell}{b}\right)\Bigg)^{2}+\vare\Bigg)\\
& =\frac{1}{b}\sum_{\ell=0}^{b-1}\lim_{n\to\infty}\frac{1}{n} \sum_{m=1}^{n}
          \Bigg(\phi'\left(\frac{R_{m-1}\,b^{-(m-1)}+\ell}{b}\right)\Bigg)^{2}+\vare\\
& =\frac{1}{b}\sum_{\ell=0}^{b-1}\int_{0}^{1}\Bigg(\phi'\left(\frac{x+\ell}{b}\right)\Bigg)^2\,\dd x+\vare
 = \sum_{\ell=0}^{b-1}\int_{\ell/b}^{(\ell+1)/b} (\phi'(y))^2 \,\dd y + \vare\\
& =\int_{0}^{1}\left(\phi'\left(y\right)\right)^2\,\dd y+\vare.
\end{align*}
Analogously, using also the inequality $\liminf_{n\to\infty} (a_n+b_n) \geq \liminf_{n\to\infty} a_n + \liminf_{n\to\infty} b_n$,
 whenever the right-hand side of the inequality is well-defined,
 we get $\PP$-almost surely
  $$\liminf_{n\to\infty} \frac{1}{n}\langle Z\rangle_n\geq \int_{0}^{1}\left(\phi'\left(y\right)\right)^2\,\dd y-\vare,$$
and since $\vare>0$ is arbitrary, this proves the assertion.
\proofend

\begin{Lem}\label{Lem_quadratic_Z_exp}
Let us consider the function $f$ defined by \eqref{help7} such that $\phi$ fulfills the hypotheses (H1), (H2) and (H3).
Then, for the quadratic variation process $(\langle Z\rangle_n)_{n\in\ZZ_+}$ defined in \eqref{Def_Z_quadratic_variation},
 we have
 \[
   \lim_{n\to\infty} \frac{1}{n}\EE(\langle Z\rangle_n) = \int_0^1 (\phi'(x))^2\,\dd x\in\RR_+.
 \]
 \end{Lem}
 
\begin{proof}
The sequence $\frac{1}{n}\langle Z\rangle_n$, $n\in\NN$, is uniformly bounded, since, 
 using \eqref{Def_Z_quadratic_variation} and \eqref{Y_estimate} with $\beta=1$, we have that 
 \[
  0\leq \frac{1}{n}\langle Z\rangle_n = \frac{1}{n} \sum_{m=1}^n \EE( Y_m^2 \mid \cF_{m-1}) \leq C^2,\qquad n\in\NN,
 \] 
where $C\in\RR_{++}$ is given by \eqref{phi_Holder}.
Hence the sequence $\frac{1}{n}\langle Z\rangle_n$, $n\in\NN$, is uniformly integrable. 
As a consequence of Lemma \ref{Lem_quadratic_Z_almost}, we also have that 
 $\frac{1}{n}\langle Z\rangle_n \distr \int_0^1 (\phi'(x))^2\,\dd x$ as $n\to\infty$.
It is known that the uniform integrability and weak convergence of $\frac{1}{n}\langle Z\rangle_n$, $n\in\NN$, 
 yield the statement of Lemma \ref{Lem_quadratic_Z_exp} (see, e.g., Billingsley \cite[Theorem 5.4]{Bil1}).
\end{proof} 

Next, we present the main result of the paper.

\begin{Thm}\label{Thm_main}
Let us consider the function $f$ defined by \eqref{help7} such that $\phi$ fulfills the hypotheses (H1), (H2) and (H3).
Then the continuous $\Theta$-variation function of $f$ along the sequence of $b$-adic partitions $\Pi_{n}$, $n\in\NN$, is given by
$$\langle f \rangle^{\Theta}(t)
           = t\cdot \sqrt{ \frac{2}{\pi\ln(b)}\int_0^1 (\phi'(x))^2\,\dd x }, \qquad t\in[0,1].$$
\end{Thm}

\begin{Rem}\label{Parseval}
Assume that the conditions of Theorem \ref{Thm_main} hold.
Then $\phi'\in L^{2}([0,1])$ (see part (i) of Remark \ref{phizero}) 
 and it has Fourier coefficients $(2\pi \ii j\,c_j)_{j\in\ZZ}$ following from \eqref{phi_Fourier_series},
 the hypotheses (H1) and (H2) and Theorem 2.2 in Folland \cite{Fol}.
Using Parseval's identity (see, e.g., Folland \cite[page 79]{Fol} or part (a) of Theorem 5.5 in Chapter I of Katznelson \cite{Kat})
 and that $c_{j}=0$ for $j\in b\ZZ\setminus\{0\}$ (by hypothesis (H3)), we have
$$\int_0^1 (\phi'(x))^2\,\dd x=4\pi^2 \sum_{j\notin b\ZZ} j^2 \vert c_j\vert^2,$$
and we may rewrite the statement in Theorem \ref{Thm_main} as
$$
 \langle f \rangle^{\Theta}(t)
           = t\cdot \sqrt{\frac{8\pi}{\ln(b)} \sum_{j\notin b\ZZ} j^2 \vert c_j\vert^2}, \qquad t\in[0,1].
$$
\proofend           
\end{Rem}

\noindent{\bf Proof of Theorem \ref{Thm_main}.}
We follow the method of the proof of Theorem 1.3 in Han et al.\  \cite{HanSchZha}.

If $\phi$ is identically zero, then $f$ is identically zero as well, and in this case, the statement readily holds.
In what follows, assume that $\phi$ is not identically zero. 
In this case we have $\int_0^1 (\phi'(x))^2\,\dd x\in(0,\infty)$ by part (ii) of Remark \ref{phizero}.

{\sl Step 1.}
By Lemma \ref{Lem_quadratic_Z_exp}, we have that
 \begin{align}\label{help36}
  \frac{1}{n} \EE(\langle Z \rangle_n)
      \to  \int_0^1 (\phi'(x))^2\,\dd x
      \qquad \text{as $n\to\infty$.}
 \end{align}
By Proposition \ref{Pro_Z_martingale} and the definition of quadratic variation, $(Z_n^2 - \langle Z \rangle_n)_{n\in\ZZ_+}$ is a martingale
 with respect to the filtration $(\cF_n)_{n\in\ZZ_+}$ such that $Z_0^2 - \langle Z \rangle_0 =0$ (where $(\cF_n)_{n\in\ZZ_+}$ is introduced 
 in  Proposition \ref{Pro_Z_martingale}), and hence $\EE(Z_n^2 - \langle Z \rangle_n)=0$ for all $n\in\ZZ_+$.
Therefore, by \eqref{help36}, we have
 \[
   \frac{1}{n} \EE(Z_n^2) = \frac{1}{n} \EE(\langle Z \rangle_n) \to \int_0^1 (\phi'(x))^2\,\dd x
    \qquad \text{as $n\to\infty$.}
 \]
In particular, the sequence
 \begin{align}\label{help34}
   \Big(\frac{1}{n}\EE(Z_n^2)\Big)_{n\in\NN}\text{ is bounded.}
 \end{align}
Under the hypotheses (H1) and (H2), by Lemmas \ref{Lemma_Lipcont_1} and \ref{Lemma_Lipcont_2}, we have that $\phi$
 is Lipschitz continuous. 
Hence, by \eqref{Y_estimate} with $\beta=1$, we have that $\vert Y_k\vert\leq C$ for all $k\in\NN$.
Therefore, for any $\vare>0$, there exists an $n_\vare\in\NN$ such $Y_k^2 < n\vare^2$ for all $n\geq n_\vare$ and $k\in\NN$.
Consequently, as $n\to\infty$ we have $\PP$-almost surely
 \begin{align}\label{help_Lindeberg}
  \begin{split}
  \sum_{k=1}^n  \EE\left(  \left(\frac{1}{\sqrt{n}} Y_k\right)^2  \bbone_{\{ \vert \frac{1}{\sqrt{n}} Y_k \vert\geq \vare  \}} \bigg| \cF_{k-1} \right)
   = \frac{1}{n} \sum_{k=1}^{n_\vare-1} \EE\left(Y_k^2 \bbone_{\{  Y_k^2 \geq n\vare^2\}} \Big| \cF_{k-1}\right)
     \to 0,
 \end{split}
 \end{align}
 i.e., the conditional Lindeberg condition holds.
Using Proposition \ref{Pro_Z_martingale} with $\beta=1$, Lemma \ref{Lem_quadratic_Z_almost}, \eqref{help_Lindeberg}
 and that $\phi$ is not the identically zero function, we can see that the conditions of the central limit theorem
 for zero-mean, square-integrable martingale arrays (see, e.g., Theorem 3.2 and Corollary 3.1 in Hall and Heyde \cite{HalHey})
 are satisfied, and hence we can obtain that
  \begin{align}\label{help_Z_CLT}
  \frac{1}{\sqrt{n}} Z_n = \frac{1}{\sqrt{n}} \sum_{k=1}^n Y_k \distr \cN\left(0,  \int_0^1 (\phi'(x))^2\,\dd x \right)
  \qquad \text{as $n\to\infty$,}
  \end{align}
  where $\int_0^1 (\phi'(x))^2\,\dd x \in(0,\infty)$ by part (ii) of Remark \ref{phizero}.
Using \eqref{Y_estimate} with $\beta=1$, \eqref{help34} and \eqref{help_Z_CLT}, one can see that
 the conditions of Lemma \ref{Lem_HSZ} are satisfied for the sequence $(Z_n)_{n\in\ZZ_+}$
 with $\sigma^2:=\int_0^1 (\phi'(x))^2\,\dd x$. Consequently, we obtain that 
 \begin{align}\label{help35}
   b^n \EE\left( \Theta\left( b^{-n} \vert Z_n \vert \right) \right)
      \to \sqrt{ \frac{2}{\pi\ln(b)} \int_0^1 (\phi'(x))^2\,\dd x }
      \qquad \text{as $n\to\infty$.}
 \end{align}
 Using \eqref{help15_Theta} and \eqref{help35}, we have the statement for $t=1$.

{\sl Step 2.}
Similarly, as on page 5 in Han et al.\  \cite{HanSchZha}, one can extend the result to the case $t\in[0,1)$,
 detailed as follows.
Let $t\in[0,1)$ be fixed.
By \eqref{Theta_var_Def}, the paragraph after Definition \ref{Def_cont_Theta_variation},
 and using that $R_n$ is uniformly distributed on $\{0,1,,\ldots,b^n-1\}$, we get that
 \begin{align*}
  V^{\Theta,t}_n(f) &= \sum_{k=0}^{\lfloor tb^n\rfloor} \Theta(\vert f((k+1)b^{-n}) - f(kb^{-n})\vert)
                      =  \sum_{k=0}^{b^n-1} \Theta(\vert f((k+1)b^{-n}) - f(kb^{-n})\vert) \bbone_{[0,t]}(kb^{-n}) \\
                     &= b^n \EE\big( \Theta(\vert f((R_n+1)b^{-n}) - f(R_n b^{-n})\vert) \bbone_{[0,t]}(R_n b^{-n}) \big)
                     = b^n \EE\big( \Theta(b^{-n}\vert Z_n \vert) \bbone_{[0,t]}(R_n b^{-n}) \big)  
 \end{align*}
 for sufficiently large $n\in\NN$, where, at the last equality, we used that, by page 3 in Han et al.\ \cite{HanSchZha} (or by the proof of Lemma 3.1 in Barczy and Kern \cite{BarKer}), 
 \begin{align*}
  f((R_n+1)b^{-n}) - f(R_n b^{-n}) = b^{-n}\sum_{k=1}^n Y_k = b^{-n} Z_n \quad\text{ for all }n\in\NN.
 \end{align*}
Note that, by Lemma \ref{Lem_p_var_expression}, $\EE\big( \Theta(b^{-n}\vert Z_n \vert) \bbone_{[0,t]}(R_n b^{-n}) \big)$
 is indeed well-defined for sufficiently large $n\in\NN$.
Therefore, using \eqref{def_Theta},  for sufficiently large $n\in\NN$, we get that 
 \begin{align*}
  V^{\Theta,t}_n(f) 
    &= b^n \EE\left( \frac{b^{-n} \vert Z_n\vert}{\sqrt{-\ln(b^{-n} \vert Z_n\vert)}}
                                    \bbone_{\{ b^{-n} \vert Z_n\vert \ne 0\}} \bbone_{[0,t]}(R_n b^{-n})  \right)\\
    &= \EE\left( \frac{\vert Z_n\vert}{\sqrt{ n\ln(b) - \ln(\vert Z_n\vert)}}
                                    \bbone_{\{ \vert Z_n\vert > 0 \} }\bbone_{ \{ b^{-n} R_n\leq t \}}  \right).                                   
 \end{align*}   
In what follows, let $\delta>0$ and $m\in\NN$ be such that $b^{-m}\leq \delta$.
Further, let 
 \begin{align*}
   R_{m,n} := R_n - R_{n-m}
           = \sum_{i=1}^n U_i b^{i-1} - \sum_{i=1}^{n-m} U_i b^{i-1}
           = \sum_{i=n-m+1}^n U_i b^{i-1}\qquad\text{ for all }n\geq m.
 \end{align*}
Hence, using that $R_{n-m}$ is nonnegative, it holds that 
 \begin{align}\label{help_Han_tsmall1_01}
   \{ b^{-n} R_n\leq t \} 
      = \{ b^{-n} (R_{m,n} + R_{n-m}) \leq t \}
      \subseteq \{ b^{-n} R_{m,n} \leq t \}
      \qquad\text{ for all }n\geq m.
 \end{align}
 
{\sl Step 3.} 
Let $t\in[0,1)$ be fixed.
We show that
 \begin{align}\label{help_Han_7}
  \limsup_{n\to\infty} V^{\Theta,t}_n(f) 
     \leq  t\cdot \sqrt{\frac{2\sigma^2}{\pi\ln(b)}},
 \end{align} 
 where recall that $\sigma^2 = \int_0^1 (\phi'(x))^2\,\dd x$.
By Step 1 in the proof of Lemma \ref{Lem_HSZ}, for all $\gamma\in(0,\ln(b))$, there exists an $n_0\in\NN$ such that 
 $n\ln(b) - \ln(Cn)>n\gamma$ for all $n\geq n_0$,
 and hence $\PP$-almost surely it holds that
 \[
  \sqrt{n\ln(b) - \ln(\vert Z_n\vert)}\bbone_{\{ \vert Z_n \vert>0 \}} \geq \sqrt{n\gamma}\bbone_{\{ \vert Z_n \vert>0 \}} 
   \qquad \text{for all $n\geq n_0$,}
 \]
 where $C>0$ is given by \eqref{Y_estimate} with $\beta=1$, i.e., $\vert Y_n\vert\leq C$ for all $n\in\NN$.
Hence, if $n\geq \max(m,n_0)$, then, using that $Z_n = Z_{n-m} + \sum_{k=n-m+1}^n Y_k$ for all $n\geq m$ and \eqref{help_Han_tsmall1_01},
 we get 
 \begin{align}\label{help_Han_tsmall1_02}
 \begin{split}
 V^{\Theta,t}_n(f) 
  &\leq \EE\left( \frac{\vert Z_n\vert}{\sqrt{n\gamma}}
                                    \bbone_{\{ \vert Z_n\vert > 0 \} }\bbone_{ \{ b^{-n} R_n\leq t \}}  \right)
  = \frac{1}{\sqrt{ n\gamma }} \EE(\vert Z_n\vert \bbone_{ \{ b^{-n} R_n\leq t \}}  ) \\ 
  &\leq \frac{1}{\sqrt{ n \gamma } } \EE(\vert Z_{n-m}\vert \bbone_{ \{ b^{-n} R_n\leq t \}}  )
         + \frac{1}{\sqrt{ n \gamma } }  \EE\left(\left\vert \sum_{k=n-m+1}^n Y_k \right\vert \bbone_{ \{ b^{-n} R_n\leq t \}}  \right)\\
  &\leq  \frac{1}{\sqrt{ n \gamma } } \EE(\vert Z_{n-m}\vert \bbone_{ \{ b^{-n} R_{m,n}\leq t \}}  )       
         + \frac{1}{\sqrt{ n \gamma } }  \EE\left(\left\vert \sum_{k=n-m+1}^n Y_k \right\vert \right). 
 \end{split}         
 \end{align}  
Here, by \eqref{Y_estimate} with $\beta=1$ we have 
 \begin{align}\label{help_Han_9}
   \frac{1}{\sqrt{ n \gamma } }  \EE\left(\left\vert \sum_{k=n-m+1}^n Y_k \right\vert \right)
     &\leq \frac{1}{\sqrt{ n \gamma } } \sum_{k=n-m+1}^n \EE(\vert Y_k\vert)
     = \frac{Cm}{\sqrt{ n \gamma } } \to 0
     \qquad \text{as $n\to\infty$.}
 \end{align} 
Further, note that the random variables
 \begin{align*}
  Z_{n-m} = \sum_{k=1}^{n-m} Y_k \qquad \text{and}\qquad R_{m,n} = \sum_{i=n-m+1}^n U_i b^{i-1}
 \end{align*}
 are independent, since, for each $k\in\NN$, the random variable $Y_k$ depends on $R_k$, which depends on $U_1,\ldots,U_k$,
 and hence, for $n\geq m+1$, the random variable $Z_{n-m}$ depends on $U_1,\ldots,U_{n-m}$; 
 and $(U_i)_{i\in\NN}$ are independent.  
Therefore, by \eqref{help_Han_tsmall1_01} and \eqref{help_Han_tsmall1_02}, we obtain that 
 \begin{align*}
  \limsup_{n\to\infty} V^{\Theta,t}_n(f) 
   & \leq \frac{1}{\sqrt{\gamma}}
          \limsup_{n\to\infty} \left(  \frac{1}{\sqrt{n}} \EE(\vert Z_{n-m}\vert) \PP(b^{-n} R_{m,n}\leq t ) \right) \\
   & \leq \frac{1}{\sqrt{\gamma}}
          \limsup_{n\to\infty} \left(\sqrt{\frac{n-m}{n}} \EE\left(\left\vert \frac{1}{\sqrt{n-m}} Z_{n-m}\right\vert\right) \right)          
          \limsup_{n\to\infty} \PP(b^{-n} R_{m,n}\leq t ).
 \end{align*}
Using \eqref{help_Han_2b} with $\xi_n:= Z_n$, $n\in\ZZ_+$, and $I_n:=I:=(0,\infty)$, $n\in\NN$
 (the conditions of Lemma \ref{Lem_HSZ_aux} hold due to \eqref{help34} and \eqref{help_Z_CLT}), we have that 
 \begin{align*}
   \lim_{n\to\infty} \EE\left(\left\vert \frac{1}{\sqrt{n-m}} Z_{n-m}\right\vert\right)
      = \frac{1}{\sqrt{2\pi \sigma^2}} \int_{-\infty}^\infty
           \vert z\vert \ee^{-\frac{z^2}{2\sigma^2}}\,\dd z
      = \sqrt{\frac{2\sigma^2}{\pi}}.     
 \end{align*} 
Note that $R_{n-m}< b^{n-m}$ (due to the fact that $R_{n-m}$ is uniformly distributed on $\{0,1,\ldots,b^{n-m}-1\}$),
 and hence, by the choice of $\delta$, we have $b^{-n} R_{n-m} < b^{-m}\leq \delta$. 
Consequently, we have $b^{-n} R_n - \delta = b^{-n}(R_{m,n} + R_{n-m}) -\delta  \leq b^{-n} R_{m,n}$, 
 and it yields that
 \begin{align*}
  \limsup_{n\to\infty} V^{\Theta,t}_n(f) 
      & \leq \sqrt{\frac{2\sigma^2}{\gamma\pi}}
           \limsup_{n\to\infty} \PP(b^{-n} R_{m,n}\leq t )\\
      & \leq \sqrt{\frac{2\sigma^2}{\gamma\pi}}
           \limsup_{n\to\infty} \PP(b^{-n} R_n \leq t + \delta )= \sqrt{\frac{2\sigma^2}{\gamma\pi}} (t+\delta),     
 \end{align*}
 where the last equality follows, since $R_n$ is uniformly distributed on the set $\{0,1,\ldots,b^n-1\}$ and hence we have $ \PP(b^{-n} R_n \leq t + \delta )=\frac{\lfloor b^n(t+\delta)\rfloor}{b^n} \to t+\delta$.
By taking the limits $\gamma\uparrow \ln(b)$ and $\delta\downarrow 0$, we have \eqref{help_Han_7}.
   
{\sl Step 4.} 
Let $t\in[0,1)$ be fixed.
We show that
 \begin{align}\label{help_Han_8}
  \liminf_{n\to\infty} V^{\Theta,t}_n(f) 
     \geq  t\cdot \sqrt{\frac{2\sigma^2}{\pi\ln(b)}}.   
 \end{align} 
For $t=0$, it holds trivially, and hence in what follows we can assume that $t\in(0,1)$. 
Let $\vare>0$, $\delta\in(0,t)$ and $m\in\NN$ be such that $b^{-m}\leq \delta$.
Then for all $n\geq \max(m,\frac{1}{\vare^2})$, as follows from Step 2 in the proof of Lemma \ref{Lem_HSZ}, we have
 \begin{align*}
 \bbone_{\{ \vert\frac{1}{\sqrt{n}}Z_n\vert  \geq \vare \} }
     \sqrt{n\ln(b) - \ln(\vert Z_n\vert)} 
  \leq  \bbone_{\{ \vert\frac{1}{\sqrt{n}}Z_n\vert  \geq \vare \} }
       \sqrt{n\ln(b)}.
 \end{align*}
Recall from \eqref{Y_estimate} with $\beta=1$ that $\vert Y_k\vert\leq C$ for all $k\in\NN$.
Let $n_1\in\NN$ be such that $\vare\sqrt{n_1}\geq mC$.
If $n\geq \max(n_1,m,\frac{1}{\vare^2})$ is sufficiently large, then we get that
 \begin{align*}
  V^{\Theta,t}_n(f)
   &= \EE\left( \frac{\vert Z_n\vert}{\sqrt{ n\ln(b) - \ln(\vert Z_n\vert)}}
                                    \bbone_{\{ \vert Z_n\vert > 0 \} }\bbone_{ \{ b^{-n} R_n\leq t \}}  \right)\\
   &\geq \EE\left( \frac{\vert Z_n\vert}{\sqrt{ n\ln(b)}}
                                    \bbone_{\{ \vert \frac{1}{\sqrt{n}} Z_n\vert \geq \vare \} }\bbone_{ \{ b^{-n} R_n\leq t \}}  \right)\\              
   &\geq \EE\left( \frac{\vert Z_n\vert}{\sqrt{ n\ln(b)}}
                                    \bbone_{\{ \vert \frac{1}{\sqrt{n}} Z_n\vert \geq \vare \} }\bbone_{ \{ b^{-n} R_{m,n}\leq t - \delta \}}  \right),                                  
 \end{align*} 
 where the last inequality follows from $b^{-n} R_n - \delta \leq b^{-n} R_{m,n}$, as argued in Step 3.
Then, for each $n\geq \max(n_1,m,\frac{1}{\vare^2})$, by the reverse triangular inequality, we have that 
 \begin{align*}
  \vert Z_n\vert = \left\vert Z_{n-m} + \sum_{k=n-m+1}^n Y_k\right\vert
                 \geq \vert Z_{n-m}\vert - \left\vert \sum_{k=n-m+1}^n Y_k\right\vert
                 \geq \vert Z_{n-m}\vert - mC 
                 \geq \vert Z_{n-m}\vert - \vare\sqrt{n_1}.
 \end{align*}
Hence, if $n\geq \max(n_1,m,\frac{1}{\vare^2})$ and $\frac{1}{\sqrt{n}}\vert Z_{n-m}\vert\geq 2\vare$, then we have
 \[
  2\vare \leq \frac{1}{\sqrt{n}}\vert Z_n\vert + \vare\sqrt{\frac{n_1}{n}}
         \leq  \frac{1}{\sqrt{n}}\vert Z_n\vert + \vare,
 \]
 and hence $\frac{1}{\sqrt{n}}\vert Z_n\vert\geq \vare$ holds.
Consequently, if $n\geq \max(n_1,m,\frac{1}{\vare^2})$ is sufficiently large, then we get that   
 \begin{align*}
  V^{\Theta,t}_n(f)
    &\geq \EE\left( \frac{\vert Z_{n-m}\vert}{\sqrt{ n\ln(b)}}
                                    \bbone_{\{ \vert \frac{1}{\sqrt{n}} Z_n\vert \geq \vare \} }\bbone_{ \{ b^{-n} R_{m,n}\leq t - \delta \}}  \right) \\
    &\phantom{\geq\;}
          - \frac{1}{\sqrt{ n\ln(b)}}                           
             \EE\left( \left\vert  \sum_{k=n-m+1}^n Y_k \right\vert
             \bbone_{\{ \vert \frac{1}{\sqrt{n}} Z_n\vert \geq \vare \} }\bbone_{ \{ b^{-n} R_{m,n}\leq t - \delta \}}  \right)\\
    &\geq \EE\left( \frac{\vert Z_{n-m}\vert}{\sqrt{ n\ln(b)}}
                                    \bbone_{\{ \vert \frac{1}{\sqrt{n}} Z_{n-m}\vert \geq 2\vare \} }\bbone_{ \{ b^{-n} R_{m,n}\leq t - \delta \}}  \right) 
          - \frac{1}{\sqrt{ n\ln(b)}}                           
             \EE\left( \left\vert  \sum_{k=n-m+1}^n Y_k \right\vert \right).             
  \end{align*}
Similarly to \eqref{help_Han_9}, we have
 \begin{align}\label{help37}
   \frac{1}{\sqrt{ n\ln(b)}}                           
             \EE\left( \left\vert  \sum_{k=n-m+1}^n Y_k \right\vert \right)
             \to 0
             \qquad \text{as $n\to\infty$.}
 \end{align}
As argued in Step 3, the random variables $Z_{n-m}$ and $R_{m,n}$ are independent, hence we get 
 \begin{align}\label{help_00}
 \begin{split} 
  &\EE\left( \frac{\vert Z_{n-m}\vert}{\sqrt{ n\ln(b)}}
                                    \bbone_{\{ \vert \frac{1}{\sqrt{n}} Z_{n-m}\vert \geq 2\vare \} }\bbone_{ \{ b^{-n} R_{m,n}\leq t - \delta \}}  \right) \\
  &= \frac{1}{\sqrt{\ln(b)}} \sqrt{\frac{n-m}{m}}
        \EE\left( \left\vert\frac{1}{\sqrt{n-m}} Z_{n-m}\right\vert
                                    \bbone_{\{ \vert \frac{1}{\sqrt{n-m}} Z_{n-m}\vert \geq 2\vare \sqrt{\frac{n}{n-m}}\} }  \right)
       \PP(b^{-n} R_{m,n}\leq t - \delta).                                                               
  \end{split}     
  \end{align}
Using \eqref{help_Han_2b} with $\xi_n:=Z_{n-m}$, $n\geq m$, $I_{n}:=[2\vare \sqrt{\frac{n}{n-m}},\infty)$, $n\geq m$, and $I:=(2\vare,\infty)$
 (the conditions of Lemma \ref{Lem_HSZ_aux} hold due to \eqref{help34} and \eqref{help_Z_CLT}), we have 
 \begin{align*}
   \lim_{n\to\infty} \EE\left( \left\vert\frac{1}{\sqrt{n-m}} Z_{n-m}\right\vert
                                    \bbone_{\{ \vert \frac{1}{\sqrt{n-m}} Z_{n-m}\vert \geq 2\vare \sqrt{\frac{n}{n-m}}\} }  \right)
      = \frac{1}{\sqrt{2\pi \sigma^2}} \int_{\{ z\in\RR\,:\, \vert z\vert> 2\vare\} }
           \vert z\vert \ee^{-\frac{z^2}{2\sigma^2}}\,\dd z.
 \end{align*} 
Further, using that 
 \begin{align*}
  b^{-n}R_{m,n} \leq t-\delta \quad  \Longleftrightarrow \quad  
  b^{-n}(R_n - R_{n-m}) \leq t-\delta \quad  \Longleftrightarrow \quad   
  b^{-n}R_n  \leq t-\delta + b^{-n}R_{n-m}, 
 \end{align*} 
  and that $b^{-n}R_{n-m}$ is nonnegative, we have that $\{b^{-n}R_n\leq t-\delta\}\subseteq\{b^{-n}R_{m,n} \leq t-\delta\}$.
Since $R_n$ is uniformly distributed on $\{0,1,\ldots,b^n-1\}$, we get 
 \begin{align*} 
   \PP(b^{-n} R_{m,n}\leq t - \delta)  
      \geq \PP(b^{-n} R_n \leq t - \delta)  
      = \frac{\lfloor b^n(t-\delta)\rfloor}{b^n}
          \to t - \delta
          \qquad \text{as $n\to\infty$.}
 \end{align*}
Therefore, by \eqref{help37} and \eqref{help_00}, we get 
 \begin{align*}
 \liminf_{n\to\infty} V^{\Theta,t}_n(f)
   \geq \frac{1}{\sqrt{\ln(b)}}
        \frac{1}{\sqrt{2\pi \sigma^2}} \int_{\{ z\in\RR\,:\, {\vert z\vert>} 2\vare\} }
           \vert z\vert \ee^{-\frac{z^2}{2\sigma^2}}\,\dd z
          \cdot (t-\delta).
  \end{align*} 
By taking the limits $\vare\downarrow 0$ and $\delta\downarrow 0$, we obtain \eqref{help_Han_8}, as desired.

Finally, \eqref{help_Han_7} and \eqref{help_Han_8} yield the statement.
\proofend

\section{Examples}\label{Section_examples}

First, we apply Theorem \ref{Thm_main} to the Takagi function and recover a result of Han et al.\ \cite[Theorem 1.2]{HanSchZha} 
 (see also part (i) of Theorem \ref{Thm_HSZ_Thms_12_13} for an even $b$) with a different proof.

\begin{Ex}[Takagi function]\label{Ex_Takagi}
Let us consider the function $f$ given by \eqref{help7} with the function $\phi(t):=\min_{z\in\ZZ}\vert z-t\vert$, $t\in\RR$. 
Then $\phi$ is a periodic function with period 1 and, for $t\in[0,1]$, we can also write
 \[
    \phi(t) = \begin{cases}
                 t & \text{if $t\in[0,\frac{1}{2}]$,}\\
                 1-t & \text{if $t\in[\frac{1}{2},1]$,}
             \end{cases}
 \]
which shows that the hypotheses (H1) and (H2) hold. 
It can easily be shown that the Fourier coefficients $c_{j}=\int_0^1 \phi(t)\ee^{-2\pi \ii jt}\,\dd t$, $j\in\ZZ$, of $\phi$ 
 are given by 
 $$
  c_j=\begin{cases}
           -\frac{1}{\pi^2j^2} & \text{if $j$ is odd,}\\
           \frac14 & \text{if $j=0$,}\\
           0 & \text{if $j\not=0$ is even.}
          \end{cases}
 $$
Hence, the hypothesis (H3) also holds for an even $b$. 
Therefore, if $b$ is even, then, by Theorem \ref{Thm_main}, we get
 \begin{align*}
  \langle f \rangle^{\Theta}(t)
           & = t \cdot \sqrt{ \frac{2}{\pi\ln(b)}\int_0^1 (\phi'(x))^2\,\dd x }
              = t \cdot \sqrt{ \frac{2}{\pi\ln(b)}} , \qquad t\in[0,1],
 \end{align*}
 as in part (i) of Theorem \ref{Thm_HSZ_Thms_12_13}. 
We mention that, if $b$ is even, then $\sum_{k\notin b\ZZ} k^2\vert c_k \vert^2=\frac{1}{4\pi^2}<\infty$, 
 but $\sum_{k\notin b\ZZ}  \vert k c_k\vert=\infty$.
Note also that if $b$ is odd, then the hypothesis (H3) does not hold,
 so we cannot apply Theorem \ref{Thm_main}, and, in fact, if $b$ is odd, then a different $\Theta$-variation function 
 appears in part (i) of Theorem \ref{Thm_HSZ_Thms_12_13}. 
\proofend
\end{Ex}

Next, we apply Theorem \ref{Thm_main} to Weierstrass functions, by recovering a result
 of Han et al.\ \cite[Theorem 1.3]{HanSchZha} (see also part (ii) of Theorem \ref{Thm_HSZ_Thms_12_13}).

\begin{Ex}[Weierstrass functions]\label{Ex_Weierstrass}
Let us consider the function $f$ given by \eqref{help7} with  $\phi(t):=\nu \sin(2\pi t) + \varrho \cos(2\pi t) - \varrho$, $t\in\RR$, where $\nu, \varrho\in\RR$.
Then the hypotheses (H1) and (H2) hold. It can easily be shown that the Fourier coefficients 
 $c_{j}=\int_0^1 \phi(t)\ee^{-2\pi \ii j t}\,\dd t$, $j\in\ZZ$, of $\phi$ are given by 
$$c_{j}=\begin{cases}
          \frac{\varrho}{2}-\ii j\frac{\nu}{2} & \text{if $j\in\{-1,1\}$,}\\
           -\varrho & \text{if $j=0$,}\\
           0 & \text{otherwise.}
          \end{cases}$$
Hence the hypothesis (H3) also holds for every $b\in\NN\setminus\{1\}$,
 and we can apply Theorem \ref{Thm_main} to get
 \begin{align*}
  \langle f \rangle^{\Theta}(t)
           & = t \cdot \sqrt{ \frac{2}{\pi\ln(b)}\int_0^1 (\phi'(x))^2\,\dd x }
              = t \cdot \sqrt{ \frac{2}{\pi\ln(b)}  4\pi^2 \int_0^1 ( \nu\cos(2\pi t) - \varrho\sin(2\pi t) )^2\,\dd x } \\
           & = t \cdot 2\sqrt{ \frac{\pi(\nu^2 + \varrho^2)}{\ln(b)} } , \qquad t\in[0,1],
 \end{align*}
where the last equality follows by standard calculations.
\proofend
\end{Ex}

The remaining part of this section is devoted to lacunary functions (series). 
For $d\in\NN\setminus\{1\}$ and $\alpha>0$, let us consider the function $h:\RR\to\CC$,
 \[
  h(t):= \sum_{k=0}^\infty d^{-k\alpha}\ee^{2\pi\ii d^k t}, \qquad t\in\RR,
 \]
 which is a well-defined, continuous and bounded function, since 
 \[
  \vert h(t)\vert \leq \sum_{k=0}^\infty d^{-k\alpha} = \frac{1}{1-d^{-\alpha}}, \qquad t\in\RR.
 \]
The function $h$ is an example of a lacunary function (lacunary power series), see, e.g., Kahane et al.\ \cite{KahWeiWei}. 
The Fourier coefficients $\omega_{\ell}$, $\ell\in\ZZ$, of $h$ are
$$
\omega_{\ell}=\int_{0}^{1}h(t)\ee^{-2\pi \ii \ell t}\,\dd t
   =\sum_{k=0}^{\infty}d^{-k\alpha} \int_0^1\ee^{2\pi \ii (d^{k}-\ell) t}\,\dd t
=\begin{cases}
\ell^{-\alpha} & \text{if $\ell=d^{k}$, $k\in\ZZ_+$,}\\
 0 & \text{otherwise.}
\end{cases}
$$
To get a real-valued function, we consider the real and imaginary parts of $h$, where, in case of the real part, 
 we substract the value at zero to fulfill hypothesis (H1). 
This gives the functions
\begin{align}\label{def_lacunaries}
 \begin{split}
 \phi_{1}(t):= & \Re(h(t))-h(0) =\frac12\big(h(t)+\overline{h(t)}\big)-h(0)\\ 
 = & \sum_{k=0}^\infty d^{-k\alpha}\cos(2\pi d^k t)-\frac1{1-d^{-\alpha}}, \qquad t\in\RR,\\
 \phi_{2}(t):= & \Im(h(t)) =\frac1{2\ii}\big(h(t)-\overline{h(t)}\big) = \sum_{k=0}^\infty d^{-k\alpha}\sin(2\pi d^k t), \qquad t\in\RR,
 \end{split} 
\end{align}
which are called lacunary trigonometric functions (series) in Aistleitner et al.\ \cite{AisBerTic} and fulfill the hypothesis (H1). 
Here the sequence $(d^k)_{k\in\ZZ_+}$ satisfies the classical Hadamard gap condition: $d^{k+1}/d^k = d >1$ for each $k\in\ZZ_+$.
The Fourier coefficients $c_{\ell}^{(j)}=\int_0^1 \phi_{j}(t)\ee^{-2\pi \ii \ell t}\,\dd t$, $\ell\in\ZZ$, 
 of $\phi_{j}$, $j\in\{1,2\}$, are
\begin{align}\label{help_lac_Fourier_coefficients_1}
 \begin{split}
 c_{\ell}^{(1)}
  & = \int_0^1 \left( \frac12\big(h(t)+\overline{h(t)}\big)-h(0) \right)\ee^{-2\pi \ii \ell t} \,\dd t
    = \frac{1}{2}\omega_\ell + \frac{1}{2}\overline{\int_0^1 h(t) \ee^{-2\pi \ii(-\ell)t} \,\dd t } - h(0)\delta_{\ell,0}\\
  & = \frac{1}{2}( \omega_\ell + \overline{\omega_{-\ell}}) - h(0)\delta_{\ell,0}
    = \frac{1}{2}(\omega_\ell + \omega_{-\ell}) - \frac{1}{1-d^{-\alpha}}\delta_{\ell,0}\\
  & =\begin{cases}
             \frac12\,\vert\ell\vert^{-\alpha} & \text{if $\vert\ell\vert=d^{k}$, $k\in\ZZ_+$,}\\
           -\frac1{1-d^{-\alpha}} & \text{if $\ell=0$,}\\
           0 & \text{otherwise,}
          \end{cases}
 \end{split}         
\end{align}
 where $\delta_{\ell,0}:=1$ if $\ell=0$, and $\delta_{\ell,0}:=0$ if $\ell\ne 0$, and, similarly, 
\begin{align}\label{help_lac_Fourier_coefficients_2}
 \begin{split}
c_{\ell}^{(2)}
   = \frac1{2\ii} \int_0^1 (h(t) - \overline{h(t)}) \ee^{-2\pi \ii \ell t}\,\dd t 
   = \frac1{2\ii}(\omega_{\ell}-\omega_{-\ell})
   =\begin{cases}
     \frac{\sign(\ell)}{2\ii}\,\vert\ell\vert^{-\alpha} & \text{if $\vert\ell\vert=d^{k}$, $k\in\ZZ_+$,}\\
      0 & \text{otherwise.}
     \end{cases}          
 \end{split}         
\end{align}
This shows that $\phi_1$ and $\phi_2$ satisfy the hypothesis (H3) in case $b$ and $d$ are relatively prime.

\begin{Lem}\label{lacunaryHoelder}
The following assertions holds for the lacunary trigonometric functions $\phi_{j}$, $j\in\{1,2\}$, defined in \eqref{def_lacunaries}.
\begin{itemize}
\item[(i)]
 If $\alpha\in(0,1)$, then $\phi_{j}$, $j\in\{1,2\}$, is H\"older continuous with exponent $\alpha$.
\item[(ii)]
 If $\alpha=1$, then there exists a constant $C_1\in\RR_{++}$ such that 
  \[
   \vert \phi_{j}(t) - \phi_{j}(s) \vert \leq C_1 \vert t-s\vert \log_d(\vert t-s\vert^{-1})
  \]
  for all $s,t\in[0,1]$ with $0<\vert t-s\vert\leq \frac{1}{2}$, where $j\in\{1,2\}$.
  This yields that $\phi_{j}$, $j\in\{1,2\}$,  is locally H\"older continuous at any point of $(0,1)$ 
   (and hence of $\RR\setminus\ZZ$) with any positive exponent strictly less than $1$.
\item[(iii)] If $\alpha>1$, then $\phi_{j}$, $j\in\{1,2\}$, is continuously differentiable.
\end{itemize}
\end{Lem}

\begin{proof}
(i). It follows from part (ii) of Proposition 2.2 in Barczy and Kern \cite{BarKer}
 with the choices $\psi(x):=x^\alpha$, $x\in\RR_{++}$, $b:=d$,
 $\phi_1:\RR\to\RR$, $\phi_1(x):=\cos(2\pi x) - 1$, $x\in\RR$, $\phi_2:\RR\to\RR$, $\phi_2(x):=\sin(2\pi x)$, $x\in\RR$, $\gamma=1$ and $\xi_m=1$, $m\in\ZZ_+$.
Then $\psi(b^{-1}) = b^{-\alpha}> b^{-1}$ and hence  part (ii) of Proposition 2.2 in Barczy and Kern \cite{BarKer} yields that 
 $\phi_j$, $j\in\{1,2\}$, restricted onto $[0,1]$ is H\"older continuous with exponent $-\log_b(b^{-\alpha})=\alpha$.
Since $\phi_j$, $j\in\{1,2\}$, is periodic with period 1 and vanishes on $\ZZ$, by Lemma \ref{Lemma_Lipcont_2},
 we get that $\phi_j$, $j\in\{1,2\}$, is H\"older continuous (on $\RR$) with exponent $\alpha$.
 In case of $d=2$, the assertion of part (i) for the complex-valued lacunary series $h$ can be also found 
  in Stein and Shakarchi \cite[part (c) of Exercise 15 on page 91]{SteSha}.\\[1ex]
(ii). This readily follows from part (iii) of Proposition 2.2 in Barczy and Kern \cite{BarKer}.\\[1ex]
(iii). We follow the proof of Theorem 2.6 in Folland \cite{Fol}.
By \eqref{help_lac_Fourier_coefficients_1} and \eqref{help_lac_Fourier_coefficients_2}, we have 
\begin{align}\label{help_lac_Fourier_coefficients_3}
  \sum_{\ell\in\ZZ} \vert \ell c_\ell^{(j)} \vert
     = 2\sum_{k=0}^\infty d^k \frac{1}{2} d^{-k\alpha}
     = \sum_{k=0}^\infty d^{-k(\alpha-1)}
     = \frac{1}{1-d^{-(\alpha-1)}}
     <\infty, \qquad j\in\{1,2\},
 \end{align}
  and hence, by the Weierstrass M-test, the series $2\pi \sum_{k=0}^\infty d^{-k(\alpha-1)} \cos(2\pi d^k t)$, $t\in\RR$,
 and $2\pi \sum_{k=0}^\infty d^{-k(\alpha-1)} \sin(2\pi d^k t)$, $t\in\RR$, are absolutely and uniformly convergent,
 therefore they define continuous functions, which are the derivatives of $\phi_1$ and $\phi_2$, respectively.
\end{proof}

Lemma \ref{lacunaryHoelder} shows that for $\alpha\in(0,1)$, the lacunary trigonometric functions $\phi_{j}$, $j\in\{1,2\}$, 
 fulfill the hypothesis (H2$^{\ast}$), and for $\alpha>1$, they even fulfill the stronger hypothesis (H2). 
Note that for $\alpha\in(0,1]$, it is known, 
 by a classical result of Hardy \cite[Theorem 1.31]{Har}, that $\phi_{j}$, $j\in\{1,2\}$, 
  is a Weierstrass function that is continuous, but nowhere differentiable. 
Hence the hypothesis (H2) does not hold for $\phi_j$, $j\in\{1,2\}$, in case $\alpha\in(0,1]$.

\begin{Ex}[Lacunary trigonometric functions]\label{Ex_lacunary}
Let us consider the function $f$ defined by \eqref{help7}, where the function $\phi$ is one of the lacunary trigonometric functions 
 $\phi_{j}$, $j\in\{1,2\}$, given in \eqref{def_lacunaries}.
From the previous discussion, we can conclude that, if $\alpha>1$ and $b$ and $d$ are relatively prime, then
 the lacunary trigonometric function $\phi_{j}$, $j\in\{1,2\}$, fulfills the hypotheses (H1), (H2) and (H3). 
 Hence, by Theorem \ref{Thm_main} together with Remark \ref{Parseval}, if $\alpha>1$ and $b$ and $d$ are relative prime, then
 \[
        \langle f \rangle^{\Theta}(t)
           = t\cdot \sqrt{ \frac{2}{\pi\ln(b)}\int_0^1 (\phi'(x))^2\,\dd x }
           = t\cdot \sqrt{ \frac{8\pi}{\ln(b)}  \sum_{\ell\notin b\ZZ} \ell^2 \vert c_\ell\vert^2 }, \qquad t\in[0,1].
 \]
Here, similarly to \eqref{help_lac_Fourier_coefficients_3}, we have 
 \[
  \sum_{\ell\notin b\ZZ} \ell^2 \vert c_\ell \vert^2 
    = 2\sum_{k=0}^\infty d^{2k} \frac{1}{4} d^{-2k\alpha}
    = \frac12\sum_{k=0}^\infty d^{-2k(\alpha-1)}
    = \frac{1}{2(1-d^{2(1-\alpha)})}.
 \]
Consequently, if $\alpha>1$ and $b$ and $d$ are relatively prime, then the function
\[
        \langle f \rangle^{\Theta}(t)
           = t\cdot \sqrt{ \frac{4\pi}{ (1-d^{2(1-\alpha)}) \ln(b)} }, \qquad t\in[0,1],
 \] 
 is the continuous $\Theta$-variation function of a lacunary trigonometric function given by \eqref{def_lacunaries} 
 along the sequence of $b$-adic partitions.
\proofend
\end{Ex}

\vspace*{5mm}

\appendix

\vspace*{5mm}

\noindent{\bf\Large Appendix}

\section{Comparison of two canonical representations of real numbers in $[0,1)$}\label{App_A}

In this appendix, first we recall a $b$-adic representation of a real number in $[0,1)$, 
 where $b\in\NN\setminus\{1\}$, see, e.g., Vakil \cite[Theorem B.2.4]{Vak}.

\begin{Thm}\label{Thm_Vakil}
Let $b\in\NN\setminus\{1\}$.
For each real number $s\in[0,1)$, there exists a unique sequence $(a_n)_{n\in\NN}$ with values in $\{0,1,\ldots,b-1\}$ such that
 $s = \sum_{n=1}^\infty \frac{a_n}{b^n}$, and $a_n<b-1$ for infinitely many $n\in\NN$
 (i.e., there does not exist an $n_0\in\NN$ such that $a_n = b-1$ for all $n\geq n_0$).
One calls such a representation of $s$ a canonical (standard) $b$-adic representation.
\end{Thm}

If $s=0$, then $a_n=0$, $n\in\NN$, in Theorem \ref{Thm_Vakil}.
Further, by Theorem \ref{Thm_Vakil}, the mapping $[0,1)\ni s\mapsto (a_n)_{n\in\NN}$ is well-defined and bijective
 (where $(a_n)_{n\in\NN}$ is given in Theorem \ref{Thm_Vakil}).

Next, we recall a result on the comparison of two canonical representations of real numbers in $[0,1)$.
We believe that this result is well-known in the literature, but we cannot address any reference for it,
 and hence, for completeness, we provide a proof as well.

\begin{Lem}\label{Lem_comp_canonical}
Let $b\in\NN\setminus\{1\}$, and let $s,t\in[0,1)$ be given in their canonical $b$-adic representations, namely,
 \[
  s = \sum_{n=1}^\infty \frac{a_n}{b^n} \qquad \text{and} \qquad t = \sum_{n=1}^\infty \frac{c_n}{b^n}.
 \]
Let $n^*:=\min\{ n\in\NN : a_n\ne c_n \}$ with the convention $\min\emptyset:=\infty$.
Then $s\geq t$ holds if and only if one of the following two assertions holds:
 \begin{itemize}
   \item[(i)] $n^*=\infty$, i.e., $a_n=c_n$, $n\in\NN$, 
   \item[(ii)] $n^*<\infty$, $a_1=c_1,\ldots,a_{n^*-1}=c_{n^*-1}$ and $a_{n^*}>c_{n^*}$.
 \end{itemize}
Consequently, the order of $s$ and $t$ is precisely the lexicographical order of their 
 canonical $b$-adic representations. 
\end{Lem}

\begin{proof}
First, suppose that one of the following two assertions (i) and (ii) holds.
If assertion (i) holds, then, by Theorem \ref{Thm_Vakil}, we get that $s=t$.
If assertion (ii) holds, then we have that
 \begin{align*}
  &s =  \sum_{n=1}^\infty a_n b^{-n}
    =  \sum_{n=1}^{n^*-1} a_n b^{-n} + a_{n^*} b^{-n^*} + \sum_{n=n^*+1}^\infty a_n b^{-n},\\
  &t =  \sum_{n=1}^\infty c_n b^{-n}
    =  \sum_{n=1}^{n^*-1} c_{n} b^{-n} + c_{n^*} b^{-n^*} + \sum_{n=n^*+1}^\infty c_n b^{-n}.
 \end{align*}
Since $a_{n^*} > c_{n^*}$, i.e., $a_{n^*} - c_{n^*}>0$, and $a_{n^*} - c_{n^*}\in\ZZ$, we have that  $a_{n^*} - c_{n^*}\in\NN$,
 yielding that $(a_{n^*} - c_{n^*})b^{-n^*} \geq b^{-n^*}$, i.e., $a_{n^*}b^{-n^*}  - c_{n^*} b^{-n^*}  \geq b^{-n^*}$.
Further, we have that 
 \begin{align*}
  0 \leq \sum_{n=n^*+1}^\infty a_n b^{-n} \leq (b-1)\sum_{n=n^*+1}^\infty b^{-n} 
     =(b-1)\frac{b^{-(n^*+1)}}{1- b^{-1}} = b^{-n^*},
 \end{align*}
 and similarly
 $0 \leq \sum_{n=n^*+1}^\infty c_n b^{-n} \leq b^{-n^*}$.
Hence, using that $a_1=c_1,\ldots,a_{n^*-1} = c_{n^*-1}$, we get that
 \begin{align*}
  s\geq  \sum_{n=1}^{n^*-1} c_n b^{-n} + c_{n^*} b^{-n^*} + b^{-n^*} + \sum_{n=n^*+1}^\infty a_n b^{-n} \geq t, 
 \end{align*}
 where the last equality follows from the fact that 
 \[
  b^{-n^*} + \sum_{n=n^*+1}^\infty a_n b^{-n} \geq  \sum_{n=n^*+1}^\infty c_n b^{-n}
  \qquad \Longleftrightarrow \qquad
  \sum_{n=n^*+1}^\infty c_n b^{-n}  -   \sum_{n=n^*+1}^\infty a_n b^{-n} \leq b^{-n^*},
 \]
 which is satisfied, since, as we checked $\sum_{n=n^*+1}^\infty a_n b^{-n}$ and $\sum_{n=n^*+1}^\infty c_n b^{-n}$ belong to $[0,b^{-n^*}]$.

Suppose now that $s\geq t$.
If $s=t$, then Theorem \ref{Thm_Vakil} yields that $a_n=c_n$, $n\in\NN$, i.e., assertion (i) holds. 
If $s>t$, then Theorem \ref{Thm_Vakil} yields that $n^*\in\NN$, i.e., $a_1=c_1,\ldots,a_{n^*-1}=c_{n^*-1}$ and $a_{n^*}\ne c_{n^*}$.
To check that assertion (ii) holds, it remains to verify that $a_{n^*}>c_{n^*}$.
Then, since $\sum_{n=1}^{n^*-1} a_n b^{-n} = \sum_{n=1}^{n^*-1} c_n b^{-n}$, we get that 
 \begin{align}\label{help_comparison}
  \begin{split}
  s-t& = a_{n^*} b^{-n^*} + \sum_{n=n^*+1}^\infty a_n b^{-n}
        - c_{n^*} b^{-n^*} - \sum_{n=n^*+1}^\infty c_n b^{-n}  \\
     & = (a_{n^*} - c_{n^*})b^{-n^*}  + \sum_{n=n^*+1}^\infty (a_n-c_n) b^{-n}.
 \end{split}     
 \end{align}
On the contrary, let us suppose that $a_{n^*}< c_{n^*}$.
Then $a_{n^*} - c_{n^*}<0$, and, since $a_{n^*} - c_{n^*}\in\ZZ$, we have that $a_{n^*} - c_{n^*}\leq -1$,
 yielding that $(a_{n^*} - c_{n^*})b^{-n^*} \leq -b^{-n^*}$.
Further, we get that 
 \begin{align*}
  \left\vert \sum_{n=n^*+1}^\infty (a_n-c_n) b^{-n} \right\vert
     \leq \sum_{n=n^*+1}^\infty \vert a_n-c_n\vert b^{-n}
     \leq (b-1)\sum_{n=n^*+1}^\infty  b^{-n}
     = b^{-n^*}.
 \end{align*} 
Hence, if $a_{n^*}< c_{n^*}$ were true, then, by \eqref{help_comparison}, we would have $s-t\leq 0$, i.e., $s\leq t$.
Taking into account that $s\geq t$, we would have that $s=t$.
Therefore, by the uniqueness of canonical $b$-adic representation (see Theorem \ref{Thm_Vakil}), 
 we would have that $a_n=c_n$, $n\in\NN$. 
This would yield that $n^*=\infty$, leading us to a contradiction.
\end{proof}

\section{Auxiliary results on piecewise continuous and piecewise continuously differentiable functions}\label{App_B}

We believe that the following results are known in the literature, but we cannot address any reference for it,
 and hence, for completeness, we provide proofs as well. 

\begin{Lem}\label{Lemma_piececont_bound_meas}
Let $\phi:[0,1]\to \RR$ be a piecewise continuous function.
Then $\phi$ is bounded and Borel measurable.
\end{Lem}

\begin{proof}
By the definition of piecewise continuity, 
 there exist at most finitely many points $0\leq t_1<\cdots<t_k\leq1$ 
 for some $k\in\ZZ_+$ such that $\phi$ is continuous except at the points $t_1,\ldots,t_k$ 
 and the left- and right-hand limits $\phi(t_{\ell}-)$ and $\phi(t_\ell+)$ exist in $\RR$ for all $\ell\in\{1,\ldots,k\}$. 
In case $t_{1}=0$ or $t_{k}=1$, we only require the existence of $\phi(t_{1}+)$ or $\phi(t_{k}-)$ in $\RR$, respectively.
In what follows, we suppose that $0<t_1$ and $t_k<1$.
The forthcoming proof can be easily adjusted for handling the remaining cases $t_1=0$ or $t_k=1$.
For each $j\in\{1,\ldots,k\}$, let $\delta_j\in \big( 0,\min(t_j,1-t_j)\big)$ be such that 
 \[
  \vert \phi(t) - \phi(t_j-) \vert <\vare, \qquad t\in(t_j-\delta_j,t_j),
 \]
 and 
 \[
  \vert \phi(t) - \phi(t_j+) \vert <\vare, \qquad t\in(t_j,t_j+\delta_j).
 \]
Let 
 \[
   M:=\max_{j\in\{1,\ldots,k\}} \Big\{  \vert  \phi(t_j-) + \vare \vert, \vert  \phi(t_j-) - \vare \vert,
                                        \vert  \phi(t_j+) + \vare \vert,  \vert  \phi(t_j+) - \vare \vert ,
                                        \vert \phi(t_j)\vert \Big\}.
 \] 
Then, for any
 \[
  \delta\in \Big( 0 , \min_{j\in\{1,\ldots,k\}} \min\Big(\delta_j,\frac{t_j-t_{j-1}}{2}\Big) \Big),
 \]
 where $t_0:=0$, we have that
 \begin{align}\label{help_piece_1}
  \vert \phi(x)\vert \leq M,\qquad x\in\bigcup_{j=1}^k (t_j-\delta,t_j+\delta),
 \end{align}
 where the intervals $(t_j-\delta, t_j+\delta)$, $j\in\{1,\ldots,k\}$, are pairwise disjoint. 
Further, by the hypothesis, $\phi$ is continuous on the intervals $[0,t_1-\delta]$, $[t_j+\delta, t_{j+1}-\delta]$, $j\in\{1,\ldots,k-1\}$, 
 and $[t_k+\delta,1]$, and hence bounded on these intervals as well.
Therefore, taking into account \eqref{help_piece_1}, we get that $\phi$ is bounded on $[0,1]$.

For the measurability of $\phi$, let us consider the decomposition 
 \begin{align}\label{help_piece_2}
  \begin{split}
  \phi = \sum_{j=1}^k \phi \bbone_{(t_{j-1},t_j)} + \sum_{j=0}^{k+1} \phi \bbone_{\{t_j\}} 
       = \sum_{j=1}^k \phi\vert_{(t_{j-1},t_j)} \bbone_{(t_{j-1},t_j)} + \sum_{j=0}^{k+1} \phi(t_j) \bbone_{\{t_j\}},
  \end{split}     
 \end{align} 
 where $t_{k+1}:=1$.
Here, for each $j\in\{1,\ldots,k\}$, the restriction $\phi\vert_{(t_{j-1},t_j)}$ of the function $\phi$ onto the interval
 $(t_{j-1},t_j)$ is continuous (by the hypothesis), and hence Borel measurable, and the functions
 $\bbone_{(t_{j-1},t_j)}$, $j\in\{1,\ldots,k\}$, and $\bbone_{\{t_j\}}$, $j\in\{0,\ldots,k+1\}$, are Borel measurable as well.
Since a (finite) linear combination of Borel measurable functions is Borel measurable, 
 using \eqref{help_piece_2}, we have that $\phi$ is Borel measurable.
\end{proof}

\begin{Lem}\label{Lemma_Lipcont_1}
Let $\phi:[0,1]\to \RR$ be a continuous and piecewise continuously differentiable function.
Then $\phi$ is Lipschitz continuous.
\end{Lem}

\begin{proof}
Since $\phi$ is piecewise continuously differentiable, there exist at most finitely many points 
 $0<t_1<\cdots<t_k<1$ for some $k\in\ZZ_+$ such that $\phi$ is continuously differentiable 
 except at the points $t_1,\ldots,t_k$.
Let $t_0:=0$ and $t_{k+1}:=1$.
Note that, for each $\ell\in\{0,\ldots,k\}$,  $\phi'$ is continuous on $(t_\ell,t_{\ell+1})$
 and its left- and right-hand side derivatives exist at $t_{\ell}$ and $t_{\ell+1}$, respectively.
Therefore, for each $\ell\in\{0,\ldots,k\}$, there exists a constant $C_\ell\in\RR_{++}$ such that $\vert \phi'(x)\vert \leq C_\ell$ 
 for all $x\in(t_\ell,t_{\ell+1})$.
Then
 \begin{align*}
      \vert \phi'(x)\vert \leq  \max_{\ell\in\{1,\ldots,k\}} C_\ell =:L, \qquad x\in[0,1]\setminus\{t_0,\ldots,t_{k+1}\}.
 \end{align*}
For each $\ell\in\{0,\ldots,k\}$ and all $x,y\in[t_\ell,t_{\ell+1}]$, 
 by the Lagrange mean value theorem applied to the restriction of $\phi$ onto $[t_\ell,t_{\ell+1}]$ 
  (which can be indeed applied, since $\phi$ is continuous on $[t_\ell,t_{\ell+1}]$ 
 and differentiable on $(t_\ell,t_{\ell+1})$), we have that there exists a point $c$ in the open interval with endpoints $x$ and $y$ such that 
 \begin{align}\label{help_Lipcont_1}
  \vert \phi(x) - \phi(y)\vert = \vert \phi'(c) \vert \vert x-y\vert \leq L \vert x-y\vert.
 \end{align} 
If $\ell<\ell^*$ with $\ell,\ell^*\in\{0,\ldots,k\}$ and $x\in[t_\ell,t_{\ell+1}]$, $y\in[t_{\ell^*},t_{\ell^*+1}]$,
 then, by the triangle inequality, \eqref{help_Lipcont_1} and the fact that $0\leq t_1<\cdots<t_k\leq1$, 
 we have that 
 \begin{align*}
 \vert \phi(y) - \phi(x)\vert
   & = \left\vert (\phi(y) - \phi(t_{\ell^*}) )  + (\phi(t_{\ell^*}) - \phi(t_{\ell^*-1}))  + \cdots + 
                  (\phi(t_{\ell+1}) - \phi(x))  \right\vert\\
   &\leq \vert \phi(y) - \phi(t_{\ell^*}) \vert +  \vert \phi(t_{\ell^*}) - \phi(t_{\ell^*-1}) \vert 
          + \cdots + \vert \phi(t_{\ell+1}) - \phi(x) \vert \\
   &\leq L \Big( \vert y - t_{\ell^*} \vert + \vert t_{\ell^*} - t_{\ell^*-1} \vert  
                  + \cdots + \vert t_{\ell+1} - x \vert \Big)\\
   &= L \Big( (y - t_{\ell^*} ) + ( t_{\ell^*} - t_{\ell^*-1})  
                  + \cdots + (t_{\ell+1} - x) \Big)\\
   & = L (y-x) = L\vert y-x\vert.                                                  
 \end{align*}
This implies that $\phi$ is Lipschitz continuous with a Lipschitz constant $L$.
\end{proof}

\begin{Lem}\label{Lemma_Lipcont_2}
Let $\phi:\RR\to \RR$ be a periodic function with period 1 such that it vanishes on $\ZZ$ and 
 its restriction onto $[0,1]$ is H\"older continuous with exponent $\beta\in(0,1]$.
Then $\phi$ is H\"older continuous (on $\RR$) with exponent $\beta\in(0,1]$ as well.
\end{Lem}

\begin{proof}
Note that $\phi(t)=\phi(\{t\})$ for all $t\in \RR$, where $\{t\}$ denotes the fractional part of $t$.
Since the restriction of $\phi$ onto $[0,1]$ is H\"older continuous and hence continuous,
 we have that there exists a constant $K\in\RR_{++}$ such that $\vert \phi(x)\vert\leq K$, $x\in[0,1]$,
 and, using that $\phi$ is periodic with period $1$, we have that $\vert \phi(x)\vert\leq K$, $x\in\RR$.

Let $x, y \in \RR$ be such that $x<y$.
If $|y - x| \geq 1$, then 
\[
 |\phi(y) - \phi(x)| \leq |\phi(y)| + |\phi(x)| 
    \leq 2K\leq 2K |y - x|.
 \]

If $|y - x| < 1$ and $\lfloor x \rfloor = \lfloor y \rfloor$, 
 then, using that $\phi$ is H\"older continuous on $[0,1]$ with exponent $\beta\in(0,1]$, we get that 
 \[
  |\phi(y) - \phi(x)| = |\phi(\{y\}) - \phi(\{x\})| \leq L |\{y\} - \{x\}|^\beta
 \] 
 with some $L\in\RR_{++}$. 
Since $\lfloor x \rfloor = \lfloor y \rfloor$, then $\{y\} - \{x\} = y - x$, yielding that 
 $|\phi(y) - \phi(x)| \leq L |y - x|^\beta$.

If $|y - x| < 1$ and $\lfloor x \rfloor \neq \lfloor y \rfloor$, then, using also that $x < y$, 
 there exists a unique integer $\ell\in\ZZ$ such that $\ell-1<x \leq \ell \leq y<\ell+1$.
By the triangle inequality, we get that
 \[
  |\phi(y) - \phi(x)| = |(\phi(y) - \phi(\ell)) + (\phi(\ell) - \phi(x))| \leq |\phi(y) - \phi(\ell)| + |\phi(\ell) - \phi(x)|.
 \] 
Since $\phi$ is H\"older continuous on $[0,1]$ with exponent $\beta\in(0,1]$ and $\phi(\ell)=\phi(0) = \phi(1)$, we get that
 \[
  |\phi(y) - \phi(\ell)| = |\phi(\{y\}) - \phi(0)| \leq L |\{y\} - 0|^\beta = L (y - \ell)^\beta,
 \] 
 and 
 \[
   |\phi(\ell) - \phi(x)| = |\phi(1) - \phi(\{x\})| \leq L |1 - \{x\}|^\beta= L (\ell - x)^\beta.
 \]  
Hence
 \[
  |\phi(y) - \phi(x)| \leq L\left( (y - \ell)^\beta + (\ell - x)^\beta\right)
                       \leq L 2^{1-\beta}(y-x)^\beta 
                       \leq 2L \vert y-x\vert^\beta,
 \] 
 where we used the inequality $a^\beta + b^\beta \leq 2^{1-\beta}(a+b)^\beta$, $a,b\in\RR_+$
 (following from Jensen's inequality applied to the function $\RR_+\ni x\mapsto x^\beta$).
 
All in all, we have that
 \[
  |\phi(y) - \phi(x)| \leq 2(K+L)\vert y-x\vert, \qquad x,y\in\RR,
 \] 
 i.e., $\phi$ is H\"older continuous on $\RR$ with exponent $\beta$ and with a H\"older constant $2(K+L)$.
\end{proof}

\section{Auxiliary results on convergence of functionals of random variables}\label{App_C}

In the proofs of Sections \ref{section_intro} and \ref{Sec_Results},
 we need the following auxiliary results that can partially be found in Han et al.\ \cite{HanSchZha} 
 and are extended to our needs.
 
First, we prove Lemma \ref{Lem_p_var_expression}, 
which is due to Han et al.\ \cite[formula (2.2)]{HanSchZha},
but, as we already mentioned in the Introduction, we presented it under a weaker hypothesis 
(supposing the H\"older continuouity of $\phi$ instead of Lipschitz continuity)
 and clarified the well-definedness of $V^{\Theta,1}_n(f)$
 for sufficiently large \ $n\in\NN$ as well.
 
\smallskip 
 
\noindent {\bf Proof of Lemma \ref{Lem_p_var_expression}.}
By the assumptions, $\phi$ is bounded on $\RR$, i.e., there exists a $K\in\RR_{++}$ such that $\vert \phi(t)\vert\leq K$, $t\in\RR$.
Hence we have that
 \[
  \vert f(t)\vert \leq K \sum_{m=0}^\infty b^{-m}
                    = \frac{K}{1 - b^{-1}} <\infty,\qquad t\in[0,1].
 \]
Consequently, due to the Weierstrass M-test, the series in the definition of $f$ converges absolutely and uniformly on $[0,1]$,
 so the function $f$ is well-defined.
Finally, the uniform limit theorem implies that $f$ is continuous as well.

Further, since $\phi$ is H\"older continuous with exponent $\beta\in(0,1]$, then
 using \eqref{Y_estimate}, we have that $b^{-n}\left\vert \sum_{m=1}^n Y_m \right\vert\to0$ as $n\to\infty$, since
  \begin{align*}
   0\leq b^{-n}\left\vert \sum_{m=1}^n Y_m \right\vert
   \leq C b^{-n} \sum_{m=1}^n b^{m(1-\beta)} 
   = \begin{cases}
       C b^{1-\beta} \frac{b^{-n\beta} - b^{-n}}{b^{1-\beta} - 1}  & \text{if $\beta\in(0,1)$,}\\
       C nb^{-n} & \text{if $\beta=1$.}
     \end{cases}
 \end{align*}
 where the constant $C\in\RR_{++}$ is given by \eqref{phi_Holder}. 
Hence $\Theta\left( b^{-n}\left\vert \sum_{m=1}^n Y_m \right\vert \right)$, and consequently,
 taking into also account that $\Theta$ is nonnegative,
 the right-hand side of \eqref{help15_Theta} is well-defined for sufficiently large $n\in\NN$.

Formula \eqref{help15_Theta} in case of a Lipschitz continuous $\phi$ has been derived by 
 Han et al.\ \cite[formula (2.2)]{HanSchZha}, but their argument also works in the case of H\"older continuous
 function $\phi$.
\proofend
 
The next result is an extension of formula (2.3) in the proof of Lemma 2.1 in Han et al.\ \cite{HanSchZha}.

\begin{Lem}\label{Lem_HSZ_aux}
Let $(\xi_n)_{n\in\ZZ_+}$ be a sequence of random variables such that the sequence $\frac{1}{n}\EE(\xi_n^2)$, $n\in\NN$, is bounded,
 and $\frac{1}{\sqrt{n}} \xi_n$ converges in distribution as $n\to\infty$
to a normally distributed random variable with mean $0$ and variance $\sigma^2>0$. 
Then for any sequence of nondegenerate intervals $I_{n}\subseteq\RR_{+}$, $n\in\NN$, with endpoints $0\leq a_{n}< b_{n}\leq\infty$ such that $a_{n}\to a$ and $b_{n}\to b$
 for some $a,b\in[0,\infty]$ with $a<b$, we have that
    \begin{align}\label{help_Han_2b}
    \lim_{n\to\infty}
      \EE\left( \left\vert \frac{1}{\sqrt{n}}\xi_n \right\vert 
            \bbone_{ \big\{ \big\vert \frac{1}{\sqrt{n}}\xi_n \big\vert\in I_{n}\big\} }  \right)
     = \frac{1}{\sqrt{2\pi \sigma^2}} \int_{\{ z\in\RR \,: \, \vert z\vert\in I \}}
           \vert z\vert \ee^{-\frac{z^2}{2\sigma^2}}\,\dd z
 \end{align}
 with the notation $I:=(a,b)$.
\end{Lem}

\begin{proof}
Due to the hypothesis that $\frac{1}{n}\EE(\xi_n^2)$, $n\in\NN$, is bounded, 
 for any nondegenerate interval $I_{n}\subseteq\RR_+$, $n\in\NN$, the family of random variables 
 $\{\vert \frac{1}{\sqrt{n}}\xi_n \vert \bbone_{\{ \vert \frac{1}{\sqrt{n}}\xi_n\vert\in I_{n}\} }  : n\in\NN\}$
is uniformly integrable, since 
 \begin{align*}
  \sup_{n\in\NN}
   \EE\left( \left\vert  \left\vert \frac{1}{\sqrt{n}}\xi_n \right\vert 
            \bbone_{ \big\{ \big\vert \frac{1}{\sqrt{n}}\xi_n \big\vert\in I_{n}\big\} } \right\vert^2 \right)
   = \sup_{n\in\NN} \frac{1}{n}
             \EE \left( \xi_n^2 \bbone_{  \big\{ \big\vert \frac{1}{\sqrt{n}}\xi_n \big\vert\in I_{n}\big\} } \right)      
   \leq \sup_{n\in\NN} \frac{1}{n} \EE(\xi_n^2)
   <\infty.
 \end{align*} 
 By the mapping theorem (see, e.g., Billingsley \cite[Theorem 5.5]{Bil1}), we check that
  \begin{align}\label{help_Han_3b}
 \left\vert \frac{1}{\sqrt{n}}\xi_n \right\vert 
            \bbone_{ \big\{ \big\vert \frac{1}{\sqrt{n}}\xi_n \big\vert\in I_n\big\} } 
   \distr      
    \vert \zeta \vert \bbone_{ \{\vert \zeta \vert\in I\} }
    \qquad \text{as $n\to\infty$,}      
  \end{align}
 where $\zeta$ is a normally distributed random variable with mean $0$ and variance $\sigma^2$.
For each $n\in\NN$, let $h_n:\RR\to\RR$, $h_n(x):=\vert x\vert \bbone_{I_n}(\vert x\vert)$, $x\in\RR$,
 and $h:\RR\to\RR$, $h(x):=\vert x\vert \bbone_I(\vert x\vert)$, $x\in\RR$.
Further, let $S:=\big\{ x\in\RR : h_n(x_n)\nrightarrow h(x)\;\;  \text{for some real sequence $(x_n)_{n\in\NN}$ tending to $x$}\big\}$.
Note that $S\subset \{a,b\}$, since, for example, if $0\leq a<x<b$ and $x_n\to x$ as $n\to\infty$,
 then $0\leq a_n<\frac{a+x}{2}<x_n<\frac{x+b}{2}<b_n$ for large enough $n\in\NN$, 
 yielding that $h_n(x_n) = x_n \to x = h(x)$ as $n\to\infty$.
The other cases can be handled similarly.  
Since $\{a,b\}$ has Lebesgue measure $0$ and $\zeta$ is absolutely continuous, we have $\PP(\zeta\in \{a,b\})=0$,
 yielding that $\PP(\zeta\in S)=0$.
Therefore, we get \eqref{help_Han_3b} by the mapping theorem.

Consequently, the uniform integrability together with \eqref{help_Han_3b} yields  that 
 \begin{align*}
  \lim_{n\to\infty}
      \EE\left( \left\vert \frac{1}{\sqrt{n}}\xi_n \right\vert 
            \bbone_{ \big\{ \big\vert \frac{1}{\sqrt{n}}\xi_n \big\vert\in I_n\big\} }  \right)
   & = \EE\big( \vert \zeta \vert \bbone_{ \{\vert \zeta \vert\in I\} } \big)            
   =\frac{1}{\sqrt{2\pi \sigma^2}} \int_{\{ z\in\RR \,: \, \vert z\vert\in I \}}
       \vert z\vert \ee^{-\frac{z^2}{2\sigma^2}}\,\dd z,
 \end{align*}
 see, e.g., Theorem 5.4 in Billingsley \cite{Bil1}, which proves \eqref{help_Han_2b}.
\end{proof} 

The next auxiliary result is Lemma 2.1 in Han et al.\ \cite{HanSchZha}. 
We give a detailed proof because its steps will be needed in the proof of Theorem \ref{Thm_main}.

\begin{Lem}[Han et al.\  {\cite[Lemma 2.1]{HanSchZha}}]\label{Lem_HSZ}
Let $(\xi_n)_{n\in\ZZ_+}$ be a sequence of random variables such that
 \begin{itemize}
  \item[\rm (i)] $\xi_0=0$,
  \item[\rm (ii)] its increments are uniformly bounded, i.e., there exists a constant $K>0$ such that
        $\PP(\vert \xi_{n+1} - \xi_n\vert\leq K \text{ for all } n\in\ZZ_+)=1$,
  \item[\rm (iii)] $\frac{1}{\sqrt{n}} \xi_n$ converges in distribution as $n\to\infty$
               to a normally distributed random variable with mean $0$ and variance $\sigma^2>0$,
  \item[\rm (iv)] the sequence $\frac{1}{n}\EE(\xi_n^2)$, $n\in\NN$, is bounded.
 \end{itemize}
Then, for each $b\in\NN\setminus\{1\}$, we get that
 \[
    \lim_{n\to\infty} b^n\EE\Big( \Theta(b^{-n} \vert \xi_n\vert)\Big) = \sqrt{\frac{2\sigma^2}{\pi\ln(b)}},
 \]
 where the function $\Theta$ is defined in \eqref{def_Theta}.
\end{Lem}

\noindent{\bf Proof.}
{\sl Step 1.}
We show that
 \begin{align}\label{help_Han_5}
 \limsup_{n\to\infty} b^n\EE\Big( \Theta(b^{-n} \vert \xi_n\vert)\Big)  
  \leq \sqrt{\frac{2\sigma^2}{ \pi\ln(b)}}.       
 \end{align}
Using the definition of $K$ and that $\xi_0=0$, we have that $\PP$-almost surely 
 \begin{align*}
  \vert \xi_n\vert
    &= \left\vert \sum_{k=1}^{n}(\xi_k-\xi_{k-1})\right\vert 
    \leq  \sum_{k=1}^{n}\vert \xi_k-\xi_{k-1}\vert\leq Kn, \qquad n\in\NN.
 \end{align*}
Note that, for all $\gamma\in(0,\ln(b))$, there exists an $n_0\in\NN$ such that 
 \[
    n\gamma < n\ln(b) - \ln(Kn),\qquad n\geq n_0.
 \]
Indeed, the inequality $n\gamma < n\ln(b) - \ln(Kn)$ holds if and only if $\gamma < \ln(b) - \frac{\ln(Kn)}{n}$, where 
 $\frac{\ln(Kn)}{n}\to 0$ as $n\to\infty$. 
Hence, for all $\gamma\in(0,\ln(b))$, there exists an $n_0\in\NN$ such that $\PP$-almost surely we have
 \[
  \sqrt{n\ln(b) - \ln(\vert \xi_n\vert)} \bbone_{\{ \vert\xi_n\vert  > 0 \} } \geq \sqrt{n\ln(b) - \ln(Kn)} \bbone_{\{ \vert\xi_n\vert  > 0 \} } 
        \geq  \sqrt{n\gamma} \bbone_{\{ \vert\xi_n\vert  > 0 \} }, \qquad n\geq n_0.
 \]
Therefore, using \eqref{def_Theta} and by choosing $I:=I_{n}:=(0,\infty)$, $n\in\NN$, in \eqref{help_Han_2b}, we have that 
 \begin{align*}
  & \limsup_{n\to\infty} b^n\EE\Big( \Theta(b^{-n} \vert \xi_n\vert)\Big)  
    = \limsup_{n\to\infty}  b^n \EE\left(  \frac{b^{-n}\vert \xi_n\vert }{\sqrt{-\ln(b^{-n}\vert \xi_n\vert)}}
                    \bbone_{\{ b^{-n} \vert\xi_n\vert \ne 0\} }  \right)  \\
 &\qquad =\limsup_{n\to\infty} 
             \EE\left(  \frac{\vert \xi_n\vert }{\sqrt{ n\ln(b) - \ln(\vert \xi_n\vert)} }
             \bbone_{\{ b^{-n} \vert\xi_n\vert \ne 0\} }  \right) 
  \leq \limsup_{n\to\infty} 
               \EE\left(  \frac{\vert \xi_n\vert }{\sqrt{n\gamma} }
                     \bbone_{\{ \vert\xi_n\vert  > 0 \} }  \right)\\
& \qquad = \frac{1}{\sqrt{\gamma}}
        \frac{1}{\sqrt{2\pi \sigma^2}} \int_{\RR\setminus\{0\}}
           \vert z\vert \ee^{-\frac{z^2}{2\sigma^2}}\,\dd z 
     = \sqrt{\frac{2\sigma^2}{\pi \gamma}}
 \end{align*}
for all $\gamma\in(0,\ln(b))$. By taking the limit $\gamma\uparrow \ln(b)$, we get \eqref{help_Han_5}.

{\sl Step 2.}
We show that
 \begin{align}\label{help_Han_6}
  \liminf_{n\to\infty} b^n\EE\Big( \Theta(b^{-n} \vert \xi_n\vert)\Big) 
     \geq \sqrt{\frac{2\sigma^2}{\pi \ln(b)}},  
\end{align}
which together with \eqref{help_Han_5} yields the statement.
For each $n\in\NN$, let $\vare_{n}:=\frac1{\sqrt{n}}$. 
Then we have
 \begin{align*}
 \bbone_{\{ \vert\frac{1}{\sqrt{n}}\xi_n\vert  \geq \vare_{n} \} }
     \sqrt{n\ln(b) - \ln(\vert \xi_n\vert)} 
  \leq  \bbone_{\{ \vert\frac{1}{\sqrt{n}}\xi_n\vert  \geq \vare_{n} \} }
       \sqrt{n\ln(b)},
 \end{align*}
 since if $\vert \frac{1}{\sqrt{n}}\xi_n\vert\geq \vare_n$, i.e., $\vert \xi_n\vert\geq 1$, then 
 $\ln(\vert \xi_n\vert) \geq \ln(1)=0$.
Therefore, using \eqref{def_Theta} and by choosing $I_{n}:=[\vare_{n},\infty)$, $n\in\NN$, and $I:=(0,\infty)$ in \eqref{help_Han_2b}, we have that 
 \begin{align*}
  & \liminf_{n\to\infty} b^n\EE\Big( \Theta(b^{-n} \vert \xi_n\vert)\Big)  
      = \limsup_{n\to\infty}  b^n \EE\left(  \frac{b^{-n}\vert \xi_n\vert }{\sqrt{-\ln(b^{-n}\vert \xi_n\vert)}}
                    \bbone_{\{ b^{-n} \vert\xi_n\vert \ne 0\} }  \right) \\
  & = \liminf_{n\to\infty} 
             \EE\left(  \frac{\vert \xi_n\vert }{\sqrt{ n\ln(b) - \ln(\vert \xi_n\vert)} }
                     \bbone_{\{ \vert\xi_n\vert  > 0 \} }  \right) 
   \geq \liminf_{n\to\infty} 
               \EE\left(  \frac{\vert \xi_n\vert }{\sqrt{n\ln(b) - \ln(\vert \xi_n\vert)} }
                     \bbone_{\{ \vert\xi_n\vert \geq  1 \} }  \right) \\
  &\geq \liminf_{n\to\infty} 
               \EE\left(  \frac{\vert \xi_n\vert }{\sqrt{ n\ln(b)} }
                     \bbone_{\{ \vert\frac{1}{\sqrt{n}}\xi_n\vert  \geq \vare_{n} \} }  \right) 
   = \frac{1}{\sqrt{\ln(b)}}
        \frac{1}{\sqrt{2\pi \sigma^2}} \int_{\RR}
           \vert z\vert \ee^{-\frac{z^2}{2\sigma^2}}\,\dd z= \sqrt{\frac{2\sigma^2}{\pi\ln(b)}},
 \end{align*}
 yielding \eqref{help_Han_6}.
\proofend

The last auxiliary result is a generalization of Lemma 2.3 in Han et al.\ \cite{HanSchZha} to the case of a piecewise continuous function.
 We detail and extend the proof in Han et al.\  \cite{HanSchZha} to our more general case.

\begin{Lem}\label{Lem_weak_conv}
For the sequence of random variables $(R_m)_{m\in\NN}$ defined in \eqref{help33}, we have that
 \[
   \frac{1}{n}\sum_{\ell=1}^n g(b^{-\ell} R_\ell) \to \int_0^1 g(x)\,\dd x \quad \text{ as $n\to\infty$ \ $\PP$-almost surely}
 \]
 for every piecewise continuous function $g:[0,1]\to\RR$.
\end{Lem}

\noindent{\bf Proof.}
{\sl Step 1.}
We extend the sequence $(U_n)_{n\in\NN}$ of independent and identically distributed random variables on $\{0,1,\ldots, b-1\}$ 
 (introduced before the formula \eqref{help33})
 to a two-sided independent and identically distributed sequence $(U_n)_{n\in\ZZ}$. 
Then the random variables
 \[
   X_n := \sum_{j=1}^\infty U_{n+1-j} b^{-j}\distre \sum_{j=1}^\infty U_j b^{-j}, \qquad n\in\ZZ,
 \]
 are uniformly distributed on the interval $(0,1)$, see, e.g., Lemma 3.20 in Kallenberg \cite{Kal} for the proof in case $b=2$, 
 which can be easily extended to an arbitrary $b\in\NN\setminus\{1\}$.

{\sl Step 2.}
Without loss of generality, we may assume that $(U_i)_{i\in\ZZ}$ is the canonical process defined on the probability space  
 \[
   (\Omega,\cA, \PP) := \Big( \{ 0,1,\ldots,b-1\}^{\ZZ}, (2^{\{0,1,\ldots,b-1\}})^{\otimes \ZZ}, \mu^{\otimes \ZZ} \Big),
 \]  
 where $2^{\{0,1,\ldots,b-1\}}$ is the power set of $\{0,1,\ldots,b-1\}$ and $\mu$ denotes the uniform distribution on $\{0,1,\ldots,b-1\}$.
Define the shift operator $\tau:\{0,1,\ldots,b-1\}^\ZZ \to \{0,1,\ldots,b-1\}^\ZZ$ by
 \[
   \tau( (\omega_n)_{n\in\ZZ} ):= (\omega_{n+1})_{n\in\ZZ},\qquad (\omega_n)_{n\in\ZZ}\in \Omega.
 \]
Then the measure preserving dynamical system $(\Omega,\cA, \PP,\tau)$ is mixing and hence ergodic, see, e.g.,
 Example 20.26 and Remark 20.27 in Klenke \cite{Kle}.

Let us introduce the random variable $h:\Omega\to [0,1]$ defined by 
 \[
   h((\omega_n)_{n\in\ZZ}) := \sum_{j=1}^\infty \omega_{1-j} b^{-j}, \qquad (\omega_n)_{n\in\ZZ}\in\Omega.
 \]
Then $h$ is indeed a random variable (measurable), since it is a pointwise limit of measurable mappings
 (following from the fact that the coordinate mappings $U_n$, $n\in\NN$, are measurable).
Let $\tau^0$ be the identity mapping on $\{ 0,1,\ldots,b-1\}^{\ZZ}$, $\tau^n$ be the $n$-fold composition of $\tau$
 with itself, and $\tau^{-n}$ be the inverse of $\tau^n$, where $n\in\NN$.
Using that $U_n(\omega) = \omega_n = (\tau^n(\omega))_{0} = U_0(\tau^n(\omega))$ for all $\omega\in\Omega$ and $n\in\ZZ$,
 we have that
 \[
   X_n(\omega) = \sum_{j=1}^\infty  \big(\tau^{n+1-j}(\omega)\big)_{0}  b^{-j}
               = \sum_{j=1}^\infty  \big(\tau^{n}(\omega)\big)_{1-j}  b^{-j}
          = (h\circ \tau^n)(\omega),\qquad \omega\in\Omega, \quad n\in\ZZ.
 \]
In particular, we have $X_0(\omega) = (h\circ \tau^0)(\omega) = h(\omega)$ for all $\omega\in\Omega$, i.e., $X_0=h$.

{\sl Step 3.}
Let $g:[0,1]\to\RR$ be a piecewise continuous function.
Since $g$ is a bounded and Borel measurable (see Lemma \ref{Lemma_piececont_bound_meas}), by applying Birkhoff's ergodic theorem (see, e.g., 
 Klenke \cite[Theorem 20.14]{Kle}) to the integrable random variable $g\circ h:\Omega\to \RR$, we obtain that 
 \begin{align}\label{Birkhoff-appl}
  \frac{1}{n}\sum_{\ell=1}^n g(X_\ell(\omega)) 
      = \frac{1}{n}\sum_{\ell=1}^n g( (h\circ \tau^\ell)(\omega) )
      = \frac{1}{n}\sum_{\ell=1}^n ((g\circ h)\circ \tau^\ell)(\omega) 
      \to \EE(g\circ h)   
 \end{align}
 as $n\to\infty$ for $\PP$-almost every $\omega\in\Omega$.
By Steps 1 and 2, we have $h=X_0$ and $X_0$ is uniformly distributed on the interval $(0,1)$. Hence the limit in \eqref{Birkhoff-appl} takes the form
 \[
 \EE(g\circ h) = \EE(g(X_0)) = \int_0^1 g(x)\,\dd x.
 \]

{\sl Step 4.}
We show that, for any piecewise continuous function $g:[0,1]\to\RR$, 
 \begin{align}\label{help_Han_1}
  \frac{1}{n}\sum_{\ell=1}^n g(b^{-\ell}R_\ell(\omega)) 
    - \frac{1}{n}\sum_{\ell=1}^n g(X_\ell(\omega)) 
   \to 0   \qquad \text{as $n\to\infty$}   
 \end{align}
 for $\PP$-almost every $\omega\in\Omega$, which, by Step 3, yields the statement.
Using \eqref{help33}, the definition of $X_n$ and that $U_i\in\{0,\ldots,b-1\}$, $i\in\ZZ$, we get for all $n\in\NN$
\begin{align*}
X_n & =\sum_{j=1}^\infty U_{n+1-j} b^{-j}\geq\sum_{j=1}^n U_{n+1-j} b^{-j}=b^{-n}\sum_{i=1}^n U_{i} b^{i-1}=b^{-n}R_n\geq0,
\end{align*}
and
 \begin{align}\label{XRineq}
  0\leq  X_n-b^{-n}R_n  &= \sum_{j=n+1}^\infty U_{n+1-j} b^{-j}
     \leq \sum_{j=n+1}^\infty b^{-(j-1)}
     = b^{-n} \frac{b}{b-1}
     \leq 2b^{-n}.
 \end{align}
Let $I_n:=[b^{-n}R_n,X_n]$, $n\in\NN$, and for some $k\in\ZZ_+$, denote by $0\leq t_1<\cdots<t_k\leq1$ the at most finitely many exceptional points, where $g$ is not continuous. We will show that
\begin{equation}\label{BClimsup}
 \PP\left(\limsup_{n\to\infty}\{t_j\in I_n\}\right)=0\quad\text{ for all }j\in\{1,\ldots,k\},
\end{equation}
 and consequently, $\PP\left(\liminf_{n\to\infty}\{t_j\not\in I_n\}\right)=1$ for all $j\in\{1,\ldots k\}$,
 which yields that $\PP$-almost surely the number of intervals $I_n$, $n\in\NN$, that contain at least one 
 of the exceptional points $0\leq t_1<\cdots<t_k\leq 1$ is finite. 
To prove \eqref{BClimsup}, by the Borel-Cantelli lemma, 
 it is sufficient to show that $\sum_{n\in\NN}\PP(t_j\in I_n)<\infty$ for all $j\in\{1,\ldots k\}$.
 Let $j\in\{1,\ldots k\}$ be fixed.
We have $$\PP(t_j\in I_n) = \PP(b^{-n}R_n\leq t_j\leq X_n),\qquad n\in\NN.$$
If $t_j=1$, then $\PP(b^{-n}R_n\leq t_j\leq X_n)=\PP(X_n=1)=0$, $n\in\NN$, since $X_n$ is uniformly distributed on $(0,1)$ by Step 1. 
Now let $t_j\in[0,1)$ and $\sum_{i=1}^\infty u_ib^i$ be the unique canonical (standard) $b$-adic representation of $t_j$, where $u_i\in\{0,\ldots,b-1\}$, $i\in\NN$, see, e.g., Theorem B.2.4 in Vakil \cite{Vak} (or Theorem \ref{Thm_Vakil}).
A non-canonical (non-standard) version is given if there exists an $n\in\NN$ such that $u_i= b-1$ for all $i\geq n$.
Note that 
 \[
 \PP(X_n\text{ is not given in canonical form})\leq\sum_{\ell=1}^\infty \PP(U_{n+1-j}=b-1\text{ for all }j\geq\ell)=0,\qquad n\in\NN,
 \]
 since $\PP(U_{n+1-j}=b-1\text{ for all }j\geq\ell)=0$ for all $n\in\NN$ and $\ell\in\NN$
 due to the fact that $U_i$, $i\in\NN$, are independent and uniformly distributed on $\{0,\ldots,b-1\}$.
Using the comparison of two canonical $b$-adic representations of real numbers in $[0,1)$ (see Lemma \ref{Lem_comp_canonical}),
 we get 
\begin{align*}
& \{t_j\leq X_n\text{ and }X_n\text{ is given in canonical form}\} \\
& \qquad = \{U_{n+1-i}=u_i,\, i\in\NN\}
          \cup\bigcup_{m\in\NN}\Big\{U_{n+1-i}=u_i,\, i\in\{1,\ldots,m-1\}\text{ and }U_{n+1-m}>u_m\Big\}
\end{align*}
for all $n\in\NN$.
Since $b^{-n}R_n=\sum_{i=1}^nU_{n+1-i}b^{-i}$, $n\in\NN$, is a canonical form, similarly, for all $n\in\NN$, we get
\begin{align*}
\{b^{-n}R_n\leq t_j\} & =  \{U_{n+1-i}=u_i,\, i\in\{1,\ldots,n\} \text{ and}\; u_i=0,\, i\geq n+1 \}\\
                      & \phantom{=}\cup\bigcup_{m=1}^n\Big\{U_{n+1-i}=u_i,\, i\in\{1,\ldots,m-1\}\text{ and }U_{n+1-m}<u_m\Big\}\\
                      & \phantom{=}\cup \bigcup_{m=n+1}^\infty \Big\{U_{n+1-i}=u_i,\, i\in\{1,\ldots,n\},
                                    u_i=0, \, i\in\{n+1,\ldots,m-1\} \text{ and } 0< u_m\Big\}\\
                      & = \{U_{n+1-i}=u_i,\, i\in\{1,\ldots,n\} \}\\              
                      & \phantom{=}\cup\bigcup_{m=1}^n\Big\{U_{n+1-i}=u_i,\, i\in\{1,\ldots,m-1\}\text{ and }U_{n+1-m}<u_m\Big\}.
\end{align*}
Hence, for the intersection, we get
\begin{align*}
& \{b^{-n}R_n\leq t_j\leq X_n\text{ and }X_n\text{ is given in canonical form}\} \\
& \qquad \subseteq \{U_{n+1-i}=u_i,\, i\in\NN\} \cup\bigcup_{m=n+1}^\infty\Big\{U_{n+1-i}=u_i,\, i\in\{1,\ldots,m-1\}\text{ and }U_{n+1-m}>u_m\Big\}
\end{align*}
 for all $n\in\NN$.
Using that $\PP(U_{n+1-i}=u_i,\, i\in\NN)=0$, $n\in\NN$, this shows that 
\begin{align*}
\PP(t_j\in I_n) & = \PP(b^{-n}R_n\leq t_j\leq X_n)
  \leq \sum_{m=n+1}^\infty \PP( U_{n+1-i}=u_i,\, i\in\{1,\ldots,m-1\} )\\
  &= \sum_{m=n+1}^\infty b^{-(m-1)}=b^{-n}\,\frac1{1-b^{-1}},
  \qquad n\in\NN,
\end{align*}
 yielding that $\sum_{n\in\NN} \PP(t_j\in I_n)<\infty$, as desired.

As it was explained, from \eqref{BClimsup} it follows that $\PP$-almost surely the number of intervals $I_n$, $n\in\NN$, 
 that contain at least one of the exceptional points $0\leq t_1<\cdots<t_k\leq1$ is finite. 
Let $t_0:=0$ and $t_{k+1}:=1$. 
Then $\PP$-almost surely there exists an $N\in\NN$ such that  
 \begin{align}\label{help38}
  [b^{-n}R_n,X_n] \subseteq [0,t_1) \cup \bigcup_{j=1}^{k-1} (t_j,t_{j+1})\cup (t_k,1],
  \qquad n\geq N,
 \end{align}
 which yields the unique existence of an index $j_n\in\{0,\ldots, k\}$, $n\geq N$, satisfying
 $$[b^{-n}R_n,X_n]\subseteq\begin{cases}
[0,t_{1}) & \text{ if }j_n=0,\\
(t_{j_n},t_{j_{n}+1}) & \text{ if }j_n\in\{1,\ldots,k-1\},\\
(t_k,1] & \text{ if }j_n=k.\\
\end{cases}$$
For $j\in\{0,\ldots, k\}$, let the function $g_j:[t_j,t_{j+1}]\to\RR$ be given by
$$g_j(x)=\begin{cases}
g(x) & \text{ if }x\in(t_j,t_{j+1}),\\
g(t_j+) & \text{ if }x=t_j,\\
g(t_{j+1}-) & \text{ if }x=t_{j+1}.\\
\end{cases}$$
Then $g_j$, $j\in\{0,\ldots, k\}$, is uniformly continuous.
Hence, for each $j\in\{0,\ldots, k\}$ and all $\vare>0$, there exists a $\delta_j>0$ such that 
 \[
   \vert g_j(x) - g_j(y)\vert <\vare \qquad \text{for $\vert x-y\vert<\delta_j$, $x,y\in[t_j,t_{j+1}]$.}
 \]
Let $\delta:=\min_{j\in\{0,\ldots,k\}} \delta_j$.
Since $\vert b^{-\ell}R_\ell - X_\ell\vert \leq 2b^{-\ell}$, $\ell\in\NN$ (see \eqref{XRineq}), using \eqref{help38},
 we have that $\PP$-almost surely for all $\vare>0$,
 there exists an $\ell_\vare\geq N$, $\ell_\vare\in\NN$ such that $\vert b^{-\ell}R_\ell - X_\ell\vert < \delta$ for $\ell\geq \ell_\vare$, and hence
 \[
 \vert g_j(b^{-\ell}R_\ell) - g_j(X_\ell) \vert <\vare, \qquad \ell\geq \ell_\vare, \quad \ell\in\NN,\quad j\in\{0,\ldots,\ell\}.
 \] 
Since $b^{-\ell}R_\ell$ and $X_\ell$ are not exceptional points for $\ell\geq \ell_\vare$ $\PP$-almost surely,
 and $g_j$ and $g$ coincide on $(t_j,t_{j+1})$, we have that 
 \[
   \vert g(b^{-\ell}R_\ell) - g(X_\ell) \vert <\vare, \qquad \ell\geq \ell_\vare, \quad \ell\in\NN.
 \] 
Using that $g$ is bounded (see Lemma \ref{Lemma_piececont_bound_meas}), there exists a $K>0$ such that $\vert g(x)\vert\leq K$ for all $x\in[0,1]$.
Therefore, for all $\vare>0$ and $n\geq \ell_\vare$ we get that  $\PP$-almost surely
 \begin{align*}
 &\left\vert \frac{1}{n}\sum_{\ell=1}^n g(b^{-\ell}R_\ell) 
    - \frac{1}{n}\sum_{\ell=1}^n g(X_\ell) \right\vert\\
 &\qquad \leq \frac{1}{n}\sum_{\ell=1}^{\ell_\vare-1} \vert  g(b^{-\ell}R_\ell) - g(X_\ell) \vert    
      + \frac{1}{n}\sum_{\ell=\ell_\vare}^n \vert  g(b^{-\ell}R_\ell) - g(X_\ell) \vert \\   
 &\qquad \leq \frac{(\ell_\vare-1)K}{n}
      + \frac{1}{n}\sum_{\ell=\ell_\vare}^n \vert  g(b^{-\ell}R_\ell) - g(X_\ell) \vert 
 \leq \frac{(\ell_\vare-1)K}{n} + \vare.    
 \end{align*}
Since $\frac{(\ell_\vare-1)K}{n}\to0$ as $n\to\infty$ and $\vare>0$ is arbitrary,
 this proves \eqref{help_Han_1}, as desired.
\proofend

\end{document}